\documentclass[12pt]{article}
\usepackage{amsmath,amsfonts,amsthm,amssymb,mathrsfs,latexsym,paralist, xypic}
\usepackage{soul}
\usepackage{tikz,hyperref,pgf}

\usetikzlibrary{arrows,automata}
\usepackage{slashed}

\tikzstyle{V}=[fill=black,circle,scale=0.2, outer sep = 4pt]

\usepackage{comment}

\newtheorem{thm}{Theorem}[section]
\newtheorem{prop}[thm]{Proposition}

\theoremstyle{remark}
\newtheorem{rmk}[thm]{Remark}
\theoremstyle{example}
\newtheorem{example}[thm]{Example}
\theoremstyle{definition}
\newtheorem{defn}[thm]{Definition}

\newcommand{\C}{\mathbb{C}}

\newcommand{\T}{\mathbb{T}}

\newcommand{\BB}{\mathcal{B}}
\newcommand{\PB}{\partial\mathcal{B}}

\newcommand{\R}{\mathbb{R}}
\newcommand{\N}{\mathbb{N}}
\newcommand{\Z}{\mathbb{Z}}

\usepackage[margin=0.9in]{geometry}
\AtEndDocument{\bigskip{\small%
  \textsc{Jaeseong Heo : Department of Mathematics, Research Institute for Natural Sciences,
  Hanyang University, Seoul 04763, Republic of Korea} \par
  \textit{E-mail address}: \texttt{hjs@hanyang.ac.kr} \par
  \addvspace{\medskipamount}
\textsc{Sooran Kang : College of General Education, Chung-Ang University,
 84 Heukseok-ro, Dongjak-gu, Seoul, Republic of Korea} \par
  \textit{E-mail address}: \texttt{sooran09@cau.ac.kr}  \par
  \addvspace{\medskipamount}
  \textsc{Yongdo Lim : Department of Mathematics, Sungkyunkwan University, Suwon, Republic of Korea} \par
  \textit{E-mail address}: \texttt{ylim@skku.edu}
}}

\begin{document}

\title{Dirichlet forms and ultrametric Cantor sets associated to higher-rank graphs}
\author{Jaeseong Heo, Sooran Kang and Yongdo Lim}

\date{\today}

\maketitle

\begin{abstract}
The aim of this paper is to study the heat kernel and the jump kernel of the Dirichlet form
associated to ultrametric Cantor sets $\partial\BB_\Lambda$ that is the infinite path space
of the stationary $k$-Bratteli diagram $\BB_\Lambda$, where $\Lambda$ is a finite strongly connected $k$-graph.
The Dirichlet form which we are interested in is induced by an even spectral triple
$(C_{\operatorname{Lip}}(\PB_\Lambda), \pi_\phi, \mathcal{H}, D, \Gamma)$ and is  given by
\[
Q_s(f,g)=\frac{1}{2} \int_{\Xi} \operatorname{Tr}\big(\vert D\vert^{-s} [D,\pi_{\phi}(f)]^{\ast} [D,\pi_\phi(g)] \big) \, d\nu(\phi),
\]
where $\Xi$ is the space of choice functions on $\partial \BB_\Lambda \times \partial \BB_\Lambda$.
There are two ultrametrics, $d^{(s)}$ and $d_{w_\delta}$, on $\partial \BB_\Lambda$
which make  the infinite path space $\PB_\Lambda$ an ultrametric Cantor set.
The former $d^{(s)}$ is associated to the eigenvalues of Laplace-Beltrami operator $\Delta_s$ associated to $Q_s$,
and the latter $d_{w_\delta}$ is associated to a weight function $w_\delta$ on $\BB_\Lambda$,
where $\delta\in (0,1)$.
We show that the Perron-Frobenius measure $\mu$ on $\partial \BB_\Lambda$ has the volume doubling property
with respect to both $d^{(s)}$ and $d_{w_\delta}$ and we study the asymptotic behavior  of the heat kernel associated to $Q_s$.
Moreover, we show that the Dirichlet form $Q_s$ coincides with a Dirichlet form $\mathcal{Q}_{J_s, \mu}$
which is associated to a jump kernel $J_s$ and the measure $\mu$ on $\partial \BB_\Lambda$,
and we investigate the asymptotic behavior and moments of displacements of the process.

\end{abstract}

\tableofcontents

\section{Introduction}

Higher-rank graphs (or $k$-graphs) $\Lambda$ and their $C^*$-algebras $C^*(\Lambda)$ were introduced in \cite{KP}
as a generalization of directed graphs and Cuntz-Krieger algebras
in order to give a combinatorial model and to study abstract properties of $C^*(\Lambda)$ such as simplicity and ideal structure.
Higher-rank graphs and the associated $C^*$-algebras have extensively been studied in last two decades
on the classification theory, K-theory of $C^*(\Lambda)$, and the study of spectral triples and of KMS states.
(See \cite{CKSS, ruiz-sims-sorensen, Evans, FGJKP-monic, FGJKP2, aHLRS3} and references therein).
The $C^*$-theory of Cuntz algebras and Cuntz-Krieger algebras associated to graphs (and $k$-graphs)
not only have played an important role in operator algebras and noncommutative geometry,
but also have been used in many applications in engineering and physics:
branching laws for endomorphisms, Markov measures and topological Markov chains,
wavelets and multiresolution analysis, and iterated function systems (IFS) and fractals.
(See \cite{abe-kawamura, FGKP, BJ-paths spaces, MP, dutkay-jorgensen-iterated} and references therein).

In the recent papers \cite{FGJKP1, FGJKP2}, authors have studied spectral triples and their relations
to $C^*$-representations and wavelet decompositions of the Cuntz algebra $\mathcal{O}_N$,
the Cuntz-Krieger algebra $\mathcal{O}_{A}$ of an adjacency matrix $A$ and $C^*(\Lambda)$ based on the works of \cite{PB, JS}.
In fact, Pearson-Bellissard \cite{PB} originally constructed a family of spectral triples
to study geometry of ultrametric Cantor sets induced from weighted trees.
Not only they studied $\zeta$-function, Dixmier trace and their relations to (fractal) dimensions
and (Hausdorff, probability) measures on ultrametric Cantor sets,
but also they constructed a family of Dirichlet forms from the spectral triples and studied the eigenvalues,
eigenspaces of the Laplace-Beltrami operators associated to the Dirichlet forms,
and investigated the associated diffusion on triadic Cantor sets.
Using the framework developed in \cite{PB}, Julien and Savinien \cite{JS} applied the theory
to study substitution tilings on a weighted Bratteli diagram
(which is a special kind of tree).

Although it is known in \cite{PB, JS} that the basis of eigenspaces of the Laplace-Beltrami operators
are related to wavelets (for example, Haar wavelets for the case of the triadic Cantor set \cite{PB}),
it was first noticed in \cite{FGJKP1, FGJKP2} that particular representations of $C^*$-algebras
(Cuntz algebras, Cuntz Krieger algebras and $C^*(\Lambda)$) studied in \cite{FGKP} give
``scaling and translation'' operators that relate the eigenspaces to wavelets/wavelet decompositions in general cases.
In \cite{FGJKP2}, authors associated a finite strongly connected $k$-graph $\Lambda$
to a stationary $k$-Bratteli diagram $\BB_\Lambda$ and obtained Cantor sets
from the infinite path space $\PB_\Lambda$ of $\BB_\Lambda$ and ultrametric $d_{w_\delta}$
associated to weight functions $w_\delta$ on $\PB_\Lambda$,
where $\delta\in (0,1)$.\footnote{This number $\delta\in (0,1)$ is the abscissa of convergence
of the $\zeta_\delta$-function associated to the spectral triples. See Theorem~3.8 of \cite{FGJKP2}.}
Then they \cite{FGJKP2} were also able to associate Pearson-Bellissard spectral triples to the $k$-graph $\Lambda$
under a mild hypothesis and obtain similar results to those of \cite{PB}
for $\zeta$-function, Dixmier trace  and Laplace-Beltrami operators $\Delta_s$ and their eigenspaces.
Note that the Dixmier trace induces a self-similar measure $\mu$ on $\partial \BB_\Lambda$
which agrees with the Perron-Frobenius measure $M$ introduced in \cite{aHLRS3} under a mild hypothesis.
(See Corollary~3.10 of \cite{FGJKP2}).

For the spectral triple $(C_{\operatorname{Lip}}(\PB_\Lambda), \pi_\phi, \mathcal{H}, D,\Gamma)$
\footnote{See the details in Section~\ref{sec:Dirichlet form}.},
the Dirichlet form $Q_s$ on $L^2(\PB_\Lambda, \mu)$  is given by
\[
Q_s(f,g)=\frac{1}{2} \int_{\Xi} \operatorname{Tr}\big(\vert D\vert^{-s} [D,\pi_{\phi}(f)]^{\ast} [D,\pi_\phi(g)] \big) \, d\nu(\phi),
\]
where $\Xi$ is the space of choice functions on $\partial \BB_\Lambda \times \partial \BB_\Lambda$
and $\nu$ is a measure induced from $\mu$.
Then a classical theory implies that there exists a non-positive definite self-adjoint operator $\Delta_s$,
called the Laplace-Beltrami operator, given by
\[
\langle -\Delta_s(f), g \rangle = Q_s(f,g).
\]
(See Section~\ref{sec:Dirichlet form} of this paper and Section~8.3 of \cite{PB}).
With the self-similar measure $\mu$ on $\PB_\Lambda$ (given in \eqref{eq:measure_kBD}),
we get a metric measure space $(\PB_\Lambda, \mu, d_{w_\delta})$.
Note that the Markovian semigroup or the heat kernel associated to the Dirichlet form $Q_s$
has not been investigated  in \cite{FGJKP2}.

According to Section 13 of \cite{Kig2010}, Kigami showed that the  Dirichlet form
associated to Pearson-Bellissard spectral triple of an ultrametric Cantor set
induced from a weighted tree is a special kind of the Dirichlet forms
on Cantor sets constructed in \cite{Kig2010} without a spectral triple.
In fact, the Cantor set in \cite{Kig2010} is obtained from a random walk on a tree,
and the measure and the Dirichlet form on it are given in terms of effective resistances
(see \cite[Section~2, Section~4]{Kig2010}).
Moreover, Kigami described how his theory is related to noncommutative geometry
by identifying a generalized jump kernel for the Dirichlet form of \cite{PB} and comparing his results
 ($\zeta$-function, self-similar measure, asymptotic behavior of the heat kernel)
to those of \cite{PB} when the Cantor set is given by the complete binary tree.

We think that Kigami's observation on the complete binary tree can be extended to a more general tree,
in particular, a stationary $k$-Bratteli diagram, associated to a finite strongly connected $k$-graph.
In this paper, we study the asymptotic behavior of the associated heat kernel and identify the jump kernel $J_s$
for the process associated to the Dirichlet form  $Q_s$.
It turned out that the Dirichlet form $Q_s$ coincides with the Dirichlet form $\mathcal{Q}_{J_s,\mu}$
given in Definition~\ref{def:Dirichlet-jump}, where $J_s$ is a generalized jump kernel given in \eqref{eq:J_s}
(see Proposition~\ref{prop:Dirichlet-J}).
Also, the asymptotic behavior of the heat kernel $p_t$ associated to $Q_s$
describes the jump process across the gaps of the Cantor set (see Theorem~\ref{thm:asymp-d-w}).
Note that Chen and Kumagai \cite{ChK}  have investigated the heat kernel estimates
for this kind of jump processes associated to a jump-type Dirichlet form such as $\mathcal{Q}_{J_s,\mu}$.

We begin in Section~\ref{sec:background} with definitions of $k$-graphs and  $k$-Bratteli diagrams
and discuss how we obtain ultrametric Cantor sets $\mathcal{C}$ using weighted $k$-Bratteli diagrams associated to $k$-graphs.
In Section~\ref{sec:Dirichlet form}, we first review the construction of Pearson-Bellissard's spectral triples
and Dirichlet forms $Q_s$, $s\in \R$ on $\mathcal{C}$.
Then we show in Proposition~\ref{prop:eigenvalue} that the Laplace-Beltrami operator $\Delta_s$
associated to $Q_s$ is in fact unbounded when $s<2+\delta$, where $\delta\in (0,1)$.
Moreover, we obtain the intrinsic ultrametric $d^{(s)}$ associated to the eigenvalues of $\Delta_s$ in Proposition~\ref{prop:intrinsic_metric}.
In Section~\ref{sec:heat-kernel}, we first show that the intrinsic ultrametric $d^{(s)}$
has the volume doubling property with respect to the self-similar measure $\mu$ on the infinite path space $\partial\BB_\Lambda$.
We prove our main theorem, which states that both ultrametrics $d^{(s)}$ and $d_{w_\delta}$
are  asymptotically equivalent in the sense that $d^{(s)}(x,y) \asymp (d_{w_\delta}(x,y))^{2+\delta-s}$,
where $\delta\in (0,1)$ and see \eqref{eq:defn-asymp} for `$\asymp$'.
Finally, we discuss the heat kernel and its asymptotic behavior and moments of displacement of the process
associated to the Dirichlet form $Q_s$ in Theorem~\ref{thm:asymp-d-w}.

\section*{Acknowledgement}
The first author J.H. and the third author Y.L. were supported by National Research Foundation
of Korea (NRF) grant funded by the Korea government (MEST) No.NRF-2015R1A3A2031159.
The second author S.K. was supported by Basic Science Research Program
through NRF grant funded by the Ministry of Education No.NRF-2017R1D1A1B03034697.

\section{$k$-graphs and ultrametric Cantor sets}
\label{sec:background}

\subsection{$k$-graphs and $k$-Bratteli diagrams}

Throughout this article, $\N$ denotes the non-negative integers
and $|S|$ denotes the number of elements in a set $S$ unless specified otherwise.

A \emph{directed graph} is given by a quadruple $E = (E^0, E^1, r, s)$,
where $E^0$ is the set of vertices of the graph, $E^1$ is the set of edges,
and maps $r, s: E^1 \to E^0$ denote the range and source of each edge.
A vertex $v$ in a directed graph $E$ is a \emph{source} if $r^{-1}(v) = \emptyset$.

\begin{defn}\cite[Definition~1.1]{KP}
\label{def:k-graph}
A \emph{higher-rank graph (or k-graph)} is a small category $\Lambda$ equipped with a degree functor
$d: \Lambda \to \N^k$ satisfying the \emph{factorization property}: whenever $\lambda$ is a morphism in $\Lambda$
such that $d(\lambda) = m+n$, there are unique morphisms $\mu, \nu \in \Lambda$
such that $d(\mu) = m, d(\nu) = n$, and $\lambda = \mu \nu$.
\end{defn}

We often consider $k$-graphs as a generalization of directed graphs.
For $ n \in \N^k$, we write
\[ \Lambda^n := \{ \lambda \in \Lambda: d(\lambda) = n\} ,\]
and we call morphisms $\lambda\in \Lambda^n$ \emph{paths} with degree $n\in \N^k$.
If $n=0\in \N^k$, then $\Lambda^0$ is the set of objects of $\Lambda$,
which we also refer to as the \emph{vertices} of $\Lambda$.
There are maps $r, s: \Lambda \to \Lambda^0$ which identify the range and source of each morphism, respectively.
For $v,w \in \Lambda^0$ and $n\in \N^k$, we define two sets
\[
v\Lambda^n := \{\lambda \in \Lambda^n : r(\lambda) = v\},
\quad  v\Lambda w=\{\lambda\in \Lambda : r(\lambda)=v, s(\lambda)=w\}.
\]
We say that a $k$-graph  $\Lambda$ is \emph{finite} if $|\Lambda^n | < \infty$ for all $n \in \N^k$;
$\Lambda$ \emph{has no sources} (or \emph{is source-free})
if $|v\Lambda^n |  > 0$ for all $v \in \Lambda^0$ and $n \in \N^k$;
$\Lambda$ is \emph{strongly connected} if, for all $v, w \in \Lambda^0$, $v\Lambda w \not= \emptyset$.

For each $1 \leq i \leq k$, we write $e_i$ for the standard basis vector of $\N^k$,
and define a matrix $A_i \in M_{\Lambda^0}(\N)$ by
\[
A_i(v, w) = | v \Lambda^{e_i} w |.
\]
We call $A_i$ the \emph{$i$-th vertex matrix} of $\Lambda$.
Note that the factorization property implies that the matrices $A_i$ commute, i.e. $A_i A_j=A_j A_i$ for $1\le i, j \le k$.
In fact, for a pair of composable edges $f \in v\Lambda^{e_i} u$, $g\in u\Lambda^{e_j} w$,
the factorization property implies that there exist unique edges
$\tilde{f}\in \Lambda^{e_i} w$, $\tilde{g}\in v\Lambda^{e_j}$ such that
\[
fg = \tilde{g}\tilde{f}
\]
since $d(fg)=e_i+e_j=e_j+e_i=d(\tilde{g}\tilde{f})$.

\begin{example} \label{ex:graph}
If we let $\Lambda_E$ be the category of finite paths of a directed graph $E$,
then $\Lambda_E$ is a $1$-graph with the degree functor $d:\Lambda_E\to \N$
which takes a finite path $\lambda$ to its length $|\lambda|$ (the number of edges making up $\lambda$).

Another fundamental example of a $k$-graph is the following.
For any $k\ge 1$, let $\Omega_k$ be the small category with
$\operatorname{Obj}(\Omega_k)=\N^k$ and $\operatorname{Mor}(\Omega_k)=\{(p,q)\in \N^k\times \N^k : p\le q\}$
where $p \le q$ means that each entry $p_i$ is less than and equal to $q_i$ for $1\le i \le k$.
The range and source maps $r,s:\operatorname{Mor}(\Omega_k)\to \operatorname{Obj}(\Omega_k)$ are given by $r(p,q)=p$ and $s(p,q)=q$.
If we define $d:\Omega_k\to \N^k$ by $d(p,q)=q-p$, then one can show that $(\Omega_k, d)$ is a $k$-graph.
\qed
\end{example}

According to \cite{KP}, we can define an infinite path in a $k$-graph $\Lambda$
by a $k$-graph morphism (meaning degree-preserving functor) $x:\Omega_k\to \Lambda$
and we write $\Lambda^\infty$ for the set of all infinite paths in $\Lambda$.
We endow $\Lambda^\infty$ with the topology generated by the collection of cylinder sets $\{Z(\lambda):\lambda\in\Lambda\}$,
where the cylinder set of $\lambda$ is defined by
\[
Z(\lambda)=\{\lambda x\in \Lambda^\infty\mid s(\lambda)=r(x)\}=\{y\in \Lambda^\infty\mid y(0,d(\lambda))=\lambda\}.
\footnote{Since $y\in \Lambda^\infty$, we have $y:\Omega_k\to \Lambda$ by definition.
So, $(0,d(\lambda))\in \operatorname{Mor}(\Omega_k)\subset \N^k \times \N^k$
and the image of $(0, d(\lambda))$ under $y$ should be an element of $\Lambda$ by definition.
Thus, $y(0,d(\lambda))=\lambda \in \Lambda$ means that the initial path of the infinite path $y$
with degree $d(\lambda)$ is the same as  $\lambda$.}
\]
We note that $Z(\lambda)$ is compact open for all $\lambda\in\Lambda$.
We also see that $\Lambda$ is a finite $k$-graph if and only if $\Lambda^\infty$ is compact.

Consider a family of commuting $N\times N$ matrices $\{A_1, \dots, A_k\}$ with non-negative entries.
We say that the family $\{A_1,\dots, A_k\}$ is \emph{irreducible}
if for each nonzero matrix $A_i$, there exists a finite subset $F\subset \N^k$ such that $A_F(s,t)>0$ for all $s,t\in N$,
where for $n=(n_1, \dots, n_k)\in \N^k$ and a finite subset $F$ of $\N^k$, we write
\[
A^n:=\prod_{i=1}^k A_i^{n_i}, \quad \text{and} \quad A_F:=\sum_{n\in F} A^n.
\]
By \cite[Lemma 4.1]{aHLRS3}, a $k$-graph $\Lambda$ is strongly connected
if and only if the family of vertex matrices $\{A_1, \ldots, A_k\}$ for $\Lambda$ is irreducible.
So, if a finite $k$-graph $\Lambda$ is strongly connected, then Proposition 3.1 of \cite{aHLRS3} implies that
there is a unique positive vector $\kappa^\Lambda \in (0, \infty)^{\Lambda^0}$
such that $\sum_{v \in \Lambda^0} \kappa^\Lambda_v =1$ and
\[
A_i \kappa^\Lambda = \rho_i \kappa^\Lambda \quad \text{for  all $1 \leq i \leq k$}
\]
where $\rho_i:= \rho(A_i)$ denotes the spectral radius of $A_i$.
Such an eigenvector $\kappa_\Lambda$ is called the \emph{(unimodular) Perron-Frobenius eigenvector} of $\Lambda$.
Proposition~8.1 of \cite{aHLRS3} shows that there is a unique Borel probability measure $M$
on $\Lambda^\infty$  that satisfies the self-similarity condition\footnote{For $\lambda\in \Lambda$,
$M(Z(\lambda))=\rho(\Lambda)^{-d(\lambda)} M(Z(s(\lambda)))$.} given by
\begin{equation}\label{eq:M-measure}
M(Z(\lambda))=\rho(\Lambda)^{-d(\lambda)} \kappa^\Lambda_{s(\lambda)} \quad \text{for all $\lambda\in \Lambda$},
\end{equation}
where $Z(\lambda)$ is a cylinder set of $\lambda\in \Lambda$ and $\kappa^\Lambda$ is
the unique Perron-Frobenius eigenvector of $\Lambda$.
More precisely, for $\lambda\in\Lambda$
\begin{equation}\label{eq:M-measure-2}
M(Z(\lambda))=\rho_1^{-d(\lambda)_1}\dots \rho_k^{-d(\lambda)_k} \kappa^\Lambda_{s(\lambda)},
\end{equation}
where $\rho(\Lambda)=(\rho_1,\dots,\rho_k)$ and $d(\lambda)=(d(\lambda)_1,\dots, d(\lambda)_k)$.
The measure $M$ is often called the Perron-Frobenius measure on $\Lambda^\infty$.

Now we describe Bratteli diagrams and $k$-Bratteli diagrams introduced in \cite{BJ-paths spaces, FGJKP2} as follows.

\begin{defn}\cite{BJ-paths spaces, FGJKP2}
\label{def:bratteli-diagram}
A \emph{Bratteli diagram} denoted by $\mathcal{B}$ is a directed graph with a vertex set
$\mathcal{B}^0= \bigsqcup_{n \in \N} {V}_n$, and an edge set $\mathcal{B}^1 = \bigsqcup_{n=1}^\infty {E}_n$,
where ${E}_n$ consists of edges whose source vertex lies in ${V}_{n+1}$ and whose range vertex lies in ${V}_n$.
A \textit{finite path} $\eta= \eta_1\cdots \eta_\ell$ is a finite sequence of edges
with $s(\eta_n ) = r(\eta_{n+1})$.
\end{defn}

\begin{defn}\cite[Definition~2.3]{FGJKP2}
\label{def-k-brat-diagrm-inf-path-space}
Given a Bratteli diagram $\mathcal{B}$,
the set of all infinite paths $\partial {\mathcal{B}}$ is defined by
\[
 \partial{\mathcal{B}} = \{ (x_n)_{n=1}^\infty : x_n \in {E}_n \text{ and } s(x_n) = r(x_{n+1})
 \text{ for all } n\in \N\setminus{\{0\}} \}.
\]
 For a finite path $\eta =  \eta_1 \eta_2 \cdots \eta_\ell$  of $\mathcal{B}$,
 we define the \emph{cylinder set} $[\eta]$ by
\[
[ \eta ] = \{ x =(x_n)_{n=1}^\infty \in \partial{\mathcal{B}}: x_i=\eta_i \text{ for all } \ 1 \leq i \leq \ell\}.
\]
\end{defn}

We note the collection $\mathcal{T}$ of all cylinder sets forms a compact open sub-basis
for a locally compact Hausdorff topology on $\partial{\mathcal{B}}$;
we will always consider $\partial {\mathcal{B}}$ with this topology.

\begin{defn}\cite[Definition~2.5]{FGJKP2}
\label{def-k-brat-diagrm}
Let $A_1, \ldots, A_k$ be commuting $N \times N$ matrices with non-negative integer entries.
The \emph{stationary $k$-Bratteli diagram} associated to the matrices $A_1,\ldots, A_k$
is given by a filtered set of vertices $\mathcal{B}^0 = \bigsqcup_{n\in \N} {V}_n$
and a filtered set of edges $\mathcal{B}^1 = \bigsqcup_{n=1}^\infty {E}_n$,
where the edges in ${E}_n$ go from ${V}_{n+1}$ to ${V}_{n}$, such that
\begin{itemize}
\item[(i)] for each $n \in \N$, each ${V}_n$ consists of $N$ vertices, which we will label $1,\ldots, N$;
\item[(ii)] when $ n \equiv i \pmod{k}$, there are $A_i(p,q)$ number of edges
whose range is the vertex $p$ of ${V}_n$ and whose source is the vertex $q$ of ${V}_{n+1}$.
\end{itemize}
\end{defn}

\begin{rmk} \label{rem:prob measure}
For a finite strongly connected $k$-graph $\Lambda$,
let $\mathcal{B}_\Lambda$ be  the corresponding stationary $k$-Bratteli diagram.
We first note that Proposition~2.10 of \cite{FGJKP2} implies that the infinite path space
$\Lambda^\infty$ of $\Lambda$ is homeomorphic to the infinite path space $\PB_\Lambda$ of $\BB_\Lambda$.
Thus, we can obtain the unique probability measure $\mu$ on $\partial \mathcal{B}_\Lambda$
satisfying self-similarity condition by transferring the Perron-Frobenius measure $M$ on $\Lambda^\infty$ given in \eqref{eq:M-measure} as follows.
Since we have ${V}_n\cong \Lambda^0$ for all $n\in \N$, we write $(\kappa_v)_{v\in \Lambda^0}$ (or $(\kappa_v)_{v\in V_n}$)
again for the corresponding Perron-Frobenius eigenvectors at each level of $\BB_\Lambda$.
Then the induced unique probability measure $\mu$ on $\partial \BB_\Lambda$ is given by
\begin{equation}\label{eq:measure_kBD}
\mu[\eta]=\Big(\frac{1}{\rho^{q}\rho_1\dots \rho_{t}}\Big) \kappa^\Lambda_{s(\eta)},
\end{equation}
where $\rho=\rho_1\dots \rho_k$ and $\eta$ is a finite path in $\BB_\Lambda$ with $r(\eta)\in {V}_0$
and $\vert \eta\vert=qk+t$ for some $q,t\in \N$ and $0\le t\le k-1$.
\end{rmk}

\begin{example}\label{ex1}
Here we consider two examples of $2$-graphs with one vertex. The first one is a $2$-graph with one vertex, two blue edges and two red edges and the second one is a $2$-graph with one vertex, two blue edges and one red edge.

First consider a $2$-colored graph with one vertex $v$, two blue edges $e_1, e_2$ and two red edges $f_1,f_2$:

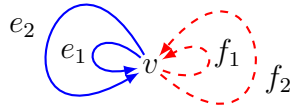
\[
\begin{tikzpicture}[scale=1.7]
 \node[inner sep=0.5pt, circle] (v) at (0,0) {$v$};
\draw[-latex, blue, thick] (v) edge [out=140, in=190, loop, min distance=15, looseness=2.5] (v);
\draw[-latex, blue, thick] (v) edge [out=120, in=210, loop, min distance=40, looseness=2.5] (v);
\draw[-latex, red, thick, dashed] (v) edge [out=-20, in=30, loop, min distance=15, looseness=2.5] (v);
\draw[-latex, red, thick, dashed] (v) edge [out=-40, in=50, loop, min distance=40, looseness=2.5] (v);
\node at (-0.6, 0.1) {$e_1$}; \node at (-1,0.3) {$e_2$}; \node at (0.6,0.1) {$f_1$}; \node at (1, -0.1) {$f_2$};
\end{tikzpicture}
\]
Then there is a $2$-graph $\Lambda$ with the above skeleton and  factorization rules given by
\begin{equation}\label{eq:ex_facto}
e_1f_1=f_1e_1\quad e_2f_2=f_2e_2, \quad e_1f_2=f_1e_2,\quad e_2f_1=f_2e_1.
\end{equation}
Then the associated vertex matrices are $A_1=(2)$ and $A_2=(2)$ and the associated stationary Bratteli diagram $\mathcal{B}_\Lambda$ is given by
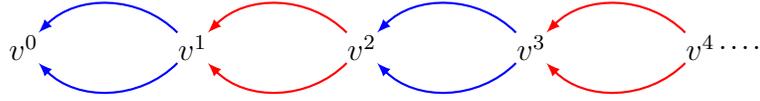
\[
\begin{tikzpicture}[scale=1.5]
\node[inner sep=0.5pt, circle] (v0) at (0,1) {$v^0$};
\node[inner sep=0.5pt, circle] (v1) at (1.5,1) {$v^1$};
\node[inner sep=0.5pt, circle] (v2) at (3,1) {$v^2$};
\node[inner sep=0.5pt, circle] (v3) at (4.5,1) {$v^3$};
\node[inner sep=0.5pt, circle] (v4) at (6,1) {$v^4$};
\draw[-latex, blue, thick] (v1) edge [out=135, in=45] (v0);
\draw[-latex, blue, thick] (v1) edge [out=225, in=-45] (v0);
\draw[-latex, red, thick] (v2) edge [out=135, in=45] (v1);
\draw[-latex, red, thick] (v2) edge [out=225, in=-45] (v1);
\draw[-latex, blue, thick] (v3) edge [out=135, in=45] (v2);
\draw[-latex, blue, thick] (v3) edge [out=225, in=-45] (v2);
\draw[-latex, red, thick] (v4) edge [out=135, in=45] (v3);
\draw[-latex, red, thick] (v4) edge [out=225, in=-45] (v3);

\node at (6.2, 1) {$.$}; \node at (6.3,1) {$.$}; \node at (6.4, 1) {$.$}; \node at (6.5,1) {$.$};
\end{tikzpicture}
\]
Since the spectral radii of both $A_1$ and $A_2$ are $2$, Proposition~2.17 of \cite{FGJKP2} implies that the infinite path space $\partial \mathcal{B}_\Lambda$ is a Cantor set. According to \cite{YD-per}, $\Lambda$ is periodic with minimal period $(1,-1)$ and the corresponding $C^*$-algebra can be identified with $C(\T)\otimes  \mathfrak{A}$, where $\mathfrak{A}$ is simple. (See the details in Theorem~3.1 and Section~5 of \cite{YD-per}).

The Perron-Frobenius eigenvector at each level of $\mathcal{B}_\Lambda$ is given by $(\kappa_v)_{v\in \Lambda^0}=1$ since there is only one vertex $v$ in $\Lambda$. Also note that any finite path $\gamma$ in the Bratteli diagram $\mathcal{B}_\Lambda$ can be represented as
\[\begin{split}
&\gamma=e_{i_1}f_{i_2} e_{i_3}f_{i_4}\dots e_{i_n}, \;\;  \quad \text{or}\;\;\quad \gamma=e_{i_1}f_{i_2} e_{i_3}f_{i_4}\dots e_{i_{m-1}}f_{i_m},\\
\end{split}\]
where $i_k\in \{1,2\}$.
The self-similar measure $\mu$ on $\partial \mathcal{B}_\Lambda$ is given by
\[
\mu([\gamma])=\Big(\frac{1}{2}\Big)^{n_1} \cdot \Big(\frac{1}{2}\Big)^{n_2} \cdot  1=\Big(\frac{1}{2}\Big)^{(n_1+n_2)},
\]
where $\gamma$ is a finite path of $\mathcal{B}_\Lambda$ with $n_1$ blue edges and $n_2$ red edges.

Now let $\Lambda'$ be the $2$-graph with one vertex $v$, two blue edges $e_1, e_2$ and one red edge $f$ that satisfies the factorization property:
\[
e_1f=fe_2 \quad \text{and} \quad fe_1=e_2f.
\]
Then the associated vertex matrices are $A_1=(2)$ and $A_2=(1)$ and the associated stationary Bratteli diagram $\mathcal{B}_{\Lambda'}$ is given by

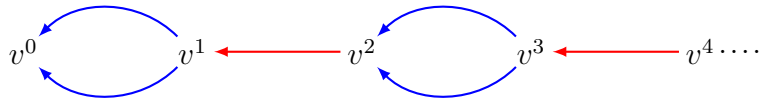
\[
\begin{tikzpicture}[scale=1.5]
\node[inner sep=0.5pt, circle] (v0) at (0,1) {$v^0$};
\node[inner sep=0.5pt, circle] (v1) at (1.5,1) {$v^1$};
\node[inner sep=0.5pt, circle] (v2) at (3,1) {$v^2$};
\node[inner sep=0.5pt, circle] (v3) at (4.5,1) {$v^3$};
\node[inner sep=0.5pt, circle] (v4) at (6,1) {$v^4$};
\draw[-latex, blue, thick] (v1) edge [out=135, in=45] (v0);
\draw[-latex, blue, thick] (v1) edge [out=225, in=-45] (v0);
\draw[-latex, red, thick] (v2) edge (v1);
\draw[-latex, blue, thick] (v3) edge [out=135, in=45] (v2);
\draw[-latex, blue, thick] (v3) edge [out=225, in=-45] (v2);
\draw[-latex, red, thick] (v4) edge (v3);

\node at (6.2, 1) {$.$}; \node at (6.3,1) {$.$}; \node at (6.4, 1) {$.$}; \node at (6.5,1) {$.$};
\end{tikzpicture}
\]

Since spectral radii of $A_1$ and $A_2$ are $\rho_1=2$ and $\rho_2=1$, and  $\rho_1\rho_2=2>1$, Proposition~2.17 of \cite{FGJKP2} implies that the infinite path space $\partial \mathcal{B}_{\Lambda'}$ is a Cantor set. As described in Example~4.10 of \cite{FGJKP-monic}, $\partial \mathcal{B}_{\Lambda'}$ is homeomorphic to $\prod_{i=1}^\infty \Z_2$. Moreover, there is a $\Lambda'$-semibranching function system associated to this $2$-graph which gives rise to a monic representation of $C^*(\Lambda')$. (See the details in Example~3.6 of \cite{FGJKP-SBFS} and Example~4.10 of \cite{FGJKP-monic}).

 The self-similar measure $\mu$ on $\partial \mathcal{B}_{\Lambda'}$ is given by
\[
\mu([\eta])=\Big(\frac{1}{2}\Big)^{m_1} \cdot \Big(1)^{m_2} \cdot  1=\Big(\frac{1}{2}\Big)^{m_1},
\]
where $\eta$ is a finite path of $\mathcal{B}_\Lambda$ with $m_1$ blue edges and $m_2$ red edges.
\qed

\end{example}

\subsection{Weights and ultrametrics}

For a Bratteli diagram (or stationary $k$-Bratteli diagram) $\BB$ and $n\in\N$, we let
\[
F^n\BB=\{\eta\in \BB : r(\eta)\in V_0,\;\; |\eta|=n\},
\]
where $|\eta|$ is the length of $\eta$, and let $F\BB=\cup_{n\in \N} F^n\BB$.

\begin{defn}\cite{PB, JS, FGJKP2}
\label{def:weight}
 A \emph{weight} on $\mathcal{B}$ is a function $w: F\mathcal{B} \to \R^+$ such that
\begin{itemize}
\item[(i)] for any vertex $v \in {V}_0, \ w(v) <1$;
\item[(ii)] $\lim_{n\to \infty} \sup \{ w(\lambda): \lambda \in F^n\mathcal{B} \} = 0$;
\item[(iii)] if $\eta$ is a sub-path of $\lambda$, then $w(\lambda) \le w(\eta)$.
\end{itemize}
\end{defn}

For $x,y\in \partial \BB$, we write $x\wedge y$ for the longest common initial sub-path of $x$ and $y$.
If $r(x)\ne r(y)$, then we define $x\wedge y=\emptyset$.

We say that a metric $d$ on a Cantor set $\mathcal{C}$ is an \emph{ultrametric}
if $d$ induces the Cantor set topology and  satisfies
\[
d(x,y ) \leq \max\{ d(x,z),d(y,z)\} \text{ for all } x,y,z \in \mathcal{C}.
\]
It is not hard to see that a weight $w$ on a Bratteli diagram gives an ultrametric $d_w$ as follows.

\begin{prop}\cite[Proposition~2.15]{FGJKP2} (c.f. \cite{PB, JS, FGJKP1})
\label{pr:weight-ultrametric}
Let $\mathcal{B}$ be a Bratteli diagram with a weight $w$.  Suppose that $\partial{\mathcal{B}}$ is a Cantor set.
The function $d_w: \partial \mathcal{B} \times \partial {\mathcal{B}} \to \R^+$ given by
\begin{equation}\label{eq:ultrametric}
d_w(x,y) = \begin{cases}
1 & \text{if}\;\; x \wedge y = \emptyset \\
0 & \text{if}\;\; x=y \\
w(x\wedge y) & \text{otherwise}
\end{cases}
\end{equation}
is an ultrametric on $\partial{\mathcal{B}}$.
Moreover, $d_w$ metrizes the cylinder set topology on $\partial \mathcal{B}$.
\end{prop}

It is shown in \cite{FGJKP2} that we can associate a $k$-graph to an ultrametric Cantor set under a mild hypothesis as follows.
Let $\Lambda$ be a finite strongly connected $k$-graph with vertex matrices $A_i$
and suppose that each spectral radius $\rho_i$ of $A_i$ is bigger than $1$.
Then Proposition~2.17 of \cite{FGJKP2} implies that  the infinite path space $\Lambda^\infty$ is a Cantor set,
and hence the infinite path space $\partial \mathcal{B}_\Lambda$
of the corresponding stationary $k$-Bratteli diagram $\mathcal{B}_\Lambda$ is also a Cantor set
since $\Lambda^\infty$  is homeomorphic to $\partial \mathcal{B}_\Lambda$.
Also, Proposition~2.19 of \cite{FGJKP2} implies that there exists a weight $\omega_\delta$
on $\mathcal{B}_\Lambda$ for $\delta\in (0,1)$ given as follows:
For any $\lambda\in F\mathcal{B}$ with $|\lambda|=n\in \N$, we write $n = qk + t$ for some $q, t \in \N_0$
with $0 \leq t \leq k-1$.  For each $\delta \in (0,1),$ we define $w_\delta: F\mathcal{B}\to \R^+$ by
\begin{equation}\label{eq:weight-k}
 w_\delta(\eta) = \left(\rho_1^{q +1} \cdots \rho_t^{q+1}
 \rho_{t+1}^q \cdots \rho_k^q \right)^{-1/\delta} \kappa^\Lambda_{s(\eta)},
\end{equation}
where $\kappa^\Lambda$ is the unimodular Perron-Frobenius eigenvector for $\Lambda$.
Moreover, Proposition~\ref{pr:weight-ultrametric} above implies
that the weight $\omega_\delta$ induces an ultrametric $d_{\omega_\delta}$ on $\partial{\mathcal{B}}$
and the metric topology induced by $d_{w_\delta}$ agrees with the cylinder set topology of $\partial \BB_\Lambda$.
Thus if we let $B_{d_{w_\delta}}(x,r)$ be the closed ball of center $x$ and radius $r>0$,
then $B_{d_{w_\delta}}(x, r)=[x_1\dots x_n]$ for some $n\in \N$, where $x=(x_n)_{n=1}^\infty \in \partial \BB_\Lambda$.
(See the details of the proof in Proposition~2.15 of \cite{FGJKP2}).

\begin{example}\label{ex1-1}

For the $2$-graph $\Lambda$ given in Example~\ref{ex1}, there is a weight $w_\delta$ on $\mathcal{B}_\Lambda$ given by
\[
w_\delta(\gamma)=2^{-\frac{n_1+n_2}{\delta}}
\]
where $\delta\in (0,1)$ and $\gamma$ is a finite path of $\mathcal{B}_\Lambda$ with $n_1$ blue edges and $n_2$ red edges.

Similarly for the $2$-graph $\Lambda'$ given in Example~\ref{ex1}, there is a weight $w'_\delta$ on $\mathcal{B}_{\Lambda'}$ given by
\[
w'_\delta(\eta)=2^{-\frac{m_1}{\delta}}
\]
where $\delta\in (0,1)$ and $\eta$ is a finite path of $\mathcal{B}_{\Lambda'}$ with $m_1$ blue edges and $m_2$ red edges.

\qed
\end{example}

\section{Spectral triples and Laplace-Beltrami operators}
\label{sec:Dirichlet form}

For the rest of this paper, we denote by $\Lambda$ and $\mathcal{B}_\Lambda$
a finite, strongly connected $k$-graph and its associated stationary $k$-Bratteli diagram
with infinite path space $\PB_\Lambda$, respectively, unless specified otherwise.

According to \cite{FGJKP2}, there exists a family of spectral triples
for the ultrametric Cantor set  associated to  $\Lambda$ under mild hypotheses.
The associated $\zeta$-function  is also regular and the measure $\mu$ on $\partial\BB_\Lambda$
induced by the Dixmier trace turned out to be equivalent to the Perron-Frobenius measure $M$ on $\Lambda^\infty$.
Moreover, the family of spectral triples induces Laplace-Beltrami operators $\Delta_s$, $s\in \R$ given via the Dirichlet forms as below.
In this section, we first review them, and then investigate the asymptotic behavior
of the eigenvalues of $\Delta_s$ and the associated ultrametric $d^{(s)}$  called the intrinsic metric.

\begin{defn}\cite{PB, JS, FGJKP2}
\label{def:spectral-triple}
An \textit{odd spectral triple} is a triple $(\mathcal{A}, \mathcal{H}, D)$
\footnote{Note that the representation $\pi$ is often included in the notation for a spectral triple.}
consisting of a Hilbert space $\mathcal{H}$,
an involutive algebra $\mathcal{A}$ of (bounded) operators on $\mathcal{H}$
and a densely defined self-adjoint operator $D$ that has compact resolvent
such that $[D,\pi(a)]$ is a bounded operator for all $a\in \mathcal{A}$,
where $\pi$ is a faithful bounded $\ast$-representation of $\mathcal{A}$ on $\mathcal{H}$.
An \textit{even spectral triple} is an odd spectral triple with a grading operator
(meaning self adjoint and unitary) $\Gamma$ on $\mathcal{H}$
such that $\Gamma D=-D\Gamma$, and $\Gamma \pi(a)=\pi(a) \Gamma$ for all $a\in \mathcal{A}$.
\end{defn}

We first suppose that $\partial \BB_\Lambda$ is a Cantor set.
For any $\delta\in (0,1)$, let $w_\delta$ be the weight on $\mathcal{B}_\Lambda$ given in \eqref{eq:weight-k}
and we denote by $d_{w_\delta}$ the induced ultrametric on $\partial\mathcal{B}_\Lambda$ as mentioned in the previous section.
Let $F\BB_\Lambda$ be the set of finite paths in $\BB_\Lambda$ whose ranges are in the first vertex set $V_0$ of $\BB_\Lambda$
(as in the beginning of section 2.2).
Suppose, in addition, that for $\lambda\in F\BB_\Lambda$, we have
\begin{equation}\label{eq:diam=w}
\operatorname{diam}[\lambda]=w_\delta(\lambda),
\end{equation}
where $\operatorname{diam}[\lambda]=\operatorname{sup}\{d_{w_\delta}(x,y)\mid x,y\in [\lambda]\}$.
Then as in \cite{FGJKP2} (c.f. \cite{PB, JS}), we can construct a family of spectral triples
for the ultrametric Cantor set $(\partial \mathcal{B}_\Lambda, d_{w_\delta})$  as follows.
\footnote{The construction of the spectral triple for $(\partial \BB_\Lambda, d_{w_\delta})$
works for a  ultrametric Cantor set induced by any weighted tree as in \cite{PB}.
Also note that the family of the spectral triples is indexed by the space of choice functions.}

A \textit{choice function} for $(\PB_\Lambda, d_{w_\delta})$ is a map $\phi:F\BB_\Lambda\to \PB_\Lambda\times \PB_\Lambda$
such that $\phi(\lambda)=(\phi_1(\lambda),\phi_2(\lambda))\in [\lambda]\times[\lambda]$ and
\begin{equation}\label{eq:choice-ft-metric}
d_{w_\delta}(\phi_1(\lambda),\phi_2(\lambda))=\operatorname{diam}[\lambda].
\end{equation}
(Hence we have $d_{w_\delta}(\phi_1(\lambda),\phi_2(\lambda))=\operatorname{diam}[\lambda]=w_\delta(\lambda)$ by \eqref{eq:diam=w}).
Here $\phi_1(\lambda)$ and $\phi_2(\lambda)$ are infinite paths in $[\lambda]$
and the subscripts 1,2  imply that the choice function gives
a pair of (distinct) infinite paths in $[\lambda]$ satisfying \eqref{eq:choice-ft-metric}.
The condition in \eqref{eq:choice-ft-metric} means that $\phi_1(\lambda)$ and $\phi_2(\lambda)$
satisfy $\phi_1(\lambda)\in [\lambda e]$ and $\phi_2(\lambda)\in [\lambda e']$ for two different edges $e,e'$.
According to \cite{PB}, the space of choice functions is the analogue of the sphere bundle of a Riemannian manifold.

We denote by $\Xi$ the space of choice functions for $(\PB_\Lambda, d_{w_\delta})$.
Since $\PB_\Lambda$ is a Cantor set, $\Xi$ is nonempty
since it implies that for every finite path $\lambda$ of $\BB_\Lambda$
we can find two distinct infinite paths $x,y\in[\lambda]$
such that $\phi_1(\lambda)=x$ and $\phi_2(\lambda)=y$. (See Proposition~2.4 of \cite{FGJKP2}).

Let $C_{\operatorname{Lip}}(\PB_\Lambda)$ be the pre-$C^*$-algebra
of Lipschitz continuous functions on $(\PB_\Lambda, d_{w_\delta})$
and let $\mathcal{H}=\ell^2(F\BB_\Lambda)\otimes \C^2$.
For $\phi\in \Xi$, we define a $\ast$-representation $\pi_\phi$
of $C_{\operatorname{Lip}}(\PB_\Lambda)$ on $\mathcal{H}$ by
\[
\pi_\phi(g) = \bigoplus_{\lambda\in F\BB_\Lambda}
\begin{pmatrix} g(\phi_1(\lambda)) & 0 \\ 0 & g(\phi_2(\lambda)) \end{pmatrix}.
\]
Then we define a \textit{Dirac-type operator} $D$ on $\mathcal{H}$ by
\[
D=\bigoplus_{\lambda\in F\BB_\Lambda} \frac{1}{w_\delta(\lambda)}\begin{pmatrix} 0 & 1 \\ 1 & 0 \end{pmatrix}.
\]
The \textit{grading operator} $\Gamma$ is given by
\[
\Gamma=1_{\ell^2(F\BB_\Lambda)}\otimes \begin{pmatrix} 1 & 0 \\ 0 & -1 \end{pmatrix}.
\]
Then one can show that $\pi_\phi$ is a faithful $\ast$-representation,
and the unbounded operator $D$ is self-adjoint with compact resolvent
and the commutator $[D,\pi_{\phi}(g)]$ is a bounded operator for all $g\in C_{\operatorname{Lip}}(\PB_\Lambda)$.
Moreover, one can check that
\[
[\Gamma, \pi_\phi(g)]=0 \quad \text{for all $g\in C_{\operatorname{Lip}}(\PB_\Lambda)$},
\]
and $\Gamma D=-D\Gamma$. Hence, we obtain an even spectral triple
$(C_{\operatorname{Lip}}(\PB_\Lambda), \pi_\phi, \mathcal{H}, D, \Gamma)$
for each choice function $\phi\in \Xi$. (See the details in \cite{PB, JS, FGJKP2}).

As defined in \cite{FGJKP2}, the associated $\zeta_\delta$-function is given by
\[
\zeta_\delta(s)=\frac{1}{2}\operatorname{Tr}(|D|^{-s})
\]
where $s\in \C$. Since we have assumed the equation \eqref{eq:diam=w},
the abscissa $s_0$ of convergence of the $\zeta_\delta$-function is $\delta$ by \cite[Theorem~3.8]{FGJKP2}.
Moreover, the associated Dixmier trace induces a measure $\mu_{w_\delta}$ on $\PB_\Lambda$
by the formula
\[
\mu_{w_\delta}[\gamma]=\mu_{w_\delta}(\chi_{\gamma})
=\lim_{s\to s_0}\frac{\operatorname{Tr}(\vert D\vert^{-s} \pi_{\phi}(\chi_\gamma))}{\operatorname{Tr}(\vert D\vert^{-s})}
\]
where $\chi_\gamma$ is the characteristic function on a cylinder set $[\gamma]$.

According to Corollary~3.10 of \cite{FGJKP2},
the induced Dixmier trace measure $\mu_{w_\delta}$ is finite and independent of $\delta$ under a mild hypothesis
\footnote{The product of the vertex matrices $A=A_1\dots A_k$ is irreducible}, so we  denote by $\mu$ without subscript.
Moreover, the measure $\mu$ agrees with the probability measure on $\PB_\Lambda$ given in \eqref{eq:measure_kBD}.
Hence, for $\lambda\in F\BB_\Lambda$ with $|\lambda|=qk+t$ $(0\le t \le k-1)$, $\mu$ is given by
\[
\mu[\lambda]=(\rho_1\dots \rho_t)^{-(q+1)} (\rho_{t+1}\dots \rho_k)^{-q}\kappa_{s(\lambda)}.
\]
(See Theorem~3.9 and Corollary~3.10 of \cite{FGJKP2} for the details).

Now we describe the associated Dirichlet form and Laplace-Beltrami operator as follows.
Recall that $F\BB_\Lambda$ is the set of finite paths in $\BB_\Lambda$ whose ranges are in the first vertex set $V_0$ of $\BB_\Lambda$.
Let $\mu$ be the induced measure on $\PB_\Lambda$ as above and let $\chi_\lambda$ be the characteristic function of the cylinder set $[\lambda]$ for $\lambda\in F\mathcal{B}_\Lambda$.
For each $s\in \R$, we define a sesquilinear form $Q_s$ on the dense subspace
$\operatorname{Dom}(Q_s):= \text{span}\{\chi_{\lambda}: \lambda \in F\BB_\Lambda\}$
of $L^2(\PB_\Lambda, \mu)$ by
\begin{equation}\label{eq:Dirichlet-spectral}
Q_s(f,g)=\frac{1}{2} \int_{\Xi} \operatorname{Tr}\big(\vert D\vert^{-s} [D,\pi_{\phi}(f)]^{\ast} [D,\pi_\phi(g)] \big) \, d\nu(\phi)
\end{equation}
for $f,g \in \operatorname{Dom}(Q_s)$,
where $\nu$ is the normalized measure on the set $\Xi$ of choice functions given by the measure $\mu$.
In particular, we have that for $\lambda\in F\BB_\Lambda$,
\begin{equation}\label{eq:measure_nu}
\nu_\lambda=\frac{\mu\times \mu}{\sum_{(e,e')\in \operatorname{ext}_1(\lambda)} \mu[\lambda e]~\mu[\lambda e']}
\end{equation}
where $\operatorname{ext}_1(\lambda)$ is the set of ordered pairs of distinct edges $(e,e')$
which extend $\gamma$ one generation further, i.e., $e\ne e'$, $r(e)=r(e')=s(\lambda)$, and $|e|=|e'|=1$.

Using the similar argument given in Section~8 of \cite{PB}, one can verify that the formula of $Q_s$ given in \eqref{eq:Dirichlet-spectral} gives rise to a closable Dirichlet form on $L^2(\partial \mathcal{B}_\Lambda, \mu)$ as follows.
 Note first that for $\lambda\in F\mathcal{B}_\Lambda$,
\[
\pi_{\psi}(\chi_\lambda)=\bigoplus_{\gamma\in F\mathcal{B}_\Lambda}\begin{pmatrix} \chi_\lambda(\psi_1(\gamma)) & 0 \\ 0 & \chi_\lambda(\psi_2(\gamma)) \end{pmatrix}=\pi_\psi(\chi_\lambda)_1 \oplus \pi_\psi(\chi_\lambda)_2 \oplus \pi_\psi(\chi_\lambda)_3,
\]
where
\[\pi_\psi(\chi_\lambda)_1 =\bigoplus_{\substack{\gamma\in F\mathcal{B}_\Lambda\\ [\lambda]\subseteq [\gamma]}}\begin{pmatrix} \chi_\lambda(\psi_1(\gamma)) & 0 \\ 0 & \chi_\lambda(\psi_2(\gamma)) \end{pmatrix},  \quad \pi_\psi(\chi_\lambda)_2=  \bigoplus_{\substack{\gamma\in F\mathcal{B}_\Lambda \\ [\gamma]\varsubsetneq [\lambda]}}\begin{pmatrix} \chi_\lambda(\psi_1(\gamma)) & 0 \\ 0 & \chi_\lambda(\psi_2(\gamma)) \end{pmatrix},
\]
\[
\pi_\psi(\chi_\lambda)_3 =\bigoplus_{\substack{\gamma\in F\mathcal{B}_\Lambda \\ \text{$[\lambda]\nsubseteq [\gamma]$ and $[\gamma]\nsubseteq [\lambda]$}}}
\begin{pmatrix} \chi_\lambda(\psi_1(\gamma)) & 0 \\ 0 & \chi_\lambda(\psi_2(\gamma)) \end{pmatrix}.
\]
Also note that
\[
D=\bigoplus_{\gamma\in F\mathcal{B}_\Lambda} \frac{1}{w_\delta(\gamma)}\begin{pmatrix} 0 & 1 \\ 1 & 0 \end{pmatrix}=D_1\oplus D_2 \oplus D_3,
\]
where
\[
D_1=\bigoplus_{\substack{\gamma\in F\mathcal{B}_\Lambda \\ [\lambda]\subseteq [\gamma]}} \frac{1}{w_\delta(\gamma)}\begin{pmatrix} 0 & 1 \\ 1 & 0 \end{pmatrix}, \quad D_2=\bigoplus_{\substack{\gamma\in F\mathcal{B}_\Lambda \\ [\gamma]\varsubsetneq [\lambda]}} \frac{1}{w_\delta(\gamma)}\begin{pmatrix} 0 & 1 \\ 1 & 0 \end{pmatrix}, \quad D_3=\bigoplus_{\substack{\gamma\in F\mathcal{B}_\Lambda \\ {\text{$[\lambda]\nsubseteq [\gamma]$ and $[\gamma]\nsubseteq [\lambda]$}}}} \frac{1}{w_\delta(\gamma)}\begin{pmatrix} 0 & 1 \\ 1 & 0 \end{pmatrix}.
\]
Then it is straightforward to see that $[D_3, \pi_\psi(\chi_\lambda)_3]=0$. Also $D_2\pi_\psi(\chi_\lambda)_2=\pi_\psi(\chi_\lambda)_2 D_2$, and  hence $[D_2, \pi_\psi(\chi_\lambda)_2]=0$. Thus
\[\begin{split}
[D,\pi_\psi (\chi_\lambda)] &=[D_1, \pi_\psi(\chi_\lambda)_1]\oplus [D_2, \pi_\psi(\chi_\lambda)_2] \oplus [D_3, \pi_\psi(\chi_\lambda)_3]\\
&=[D_1, \pi_\psi(\chi_\lambda)_1]\oplus 0 \oplus 0,
\end{split}
\]
which is finite rank. This implies that $Q_s$ given in \eqref{eq:Dirichlet-spectral} is valid for all $f,g\in \operatorname{Dom}(Q_s)$.
Now the similar arguments given in the proof of Theorem~4 of \cite[Section~8]{PB} shows that $Q_s$ is a closable Dirichlet form for all $s\in \R$ with dense domain $\operatorname{Dom}(Q_s)$.
Then Theorem~7 of \cite{PB} implies that there exists a non-positive definite self-adjoint operator $\Delta_s$ such that
\[
\langle -\Delta_s f, g \rangle=Q_s(f,g),
\]
and that $\Delta_s$ generates a Markovian semigroup. (See details in Section~8.3 of \cite{PB}).

According to \cite[Section~4.1]{JS}, one can compute the formula for $\Delta_s (\chi_\gamma)$ as follows,
and hence one can obtain the eigenvalues and the corresponding eigenvectors explicitly.
Fix a finite path $\gamma\in F \BB_\Lambda$ with $|\gamma|=n$.
For $0 \le k\le n=|\gamma|$, we write $\gamma_k=\gamma(0,k)$ for the initial sub-path of $\gamma$ with length $k$.
Then we have
\[
\Delta_s (\chi_\gamma)= -\sum_{k=0}^{n-1} \frac{1}{G_s(\gamma_k)}
\big( (\mu[\gamma_k]-\mu[\gamma_{k+1}])\chi_\gamma - \mu[\gamma](\chi_{\gamma_k}-\chi_{\gamma_{k+1}})\big),
\]
where
\[
G_s(\eta)=\frac{1}{2} w_\delta(\eta)^{2-s} \sum_{(e,e')\in \operatorname{ext}_1(\eta)} \mu[\eta e]\, \mu[\eta e']
\]
as in (4.2a), (4.2b) of \cite{JS}.
Moreover, Theorem~4.3 of \cite{JS} implies that we can show that $\Delta_s$ has a pure point spectrum,
and one can compute the eigenvalues and eigenspaces for $\Delta_s$ explicitly as follows.
 In particular, we note that the values of $\lambda_{s,\gamma}$ are non-positive and
compute the eigenvalues $\lambda_{s,\gamma}$ such that $\Delta_s \chi_\gamma = \lambda_{s,\gamma} \chi_\gamma$.
For $\gamma\in F\BB_\Lambda$,
\begin{equation}\label{eq:ev1}
\lambda_{s,\gamma}=\sum_{k=0}^{n-1}\frac{\mu[\gamma_{k+1}]-\mu[\gamma_k]}{G_s(\gamma_k)}-\frac{\mu[\gamma]}{G_s(\gamma)}
\end{equation}
is an eigenvalue with the eigenspace
\begin{equation}\label{eq:eigenspace}
E_{s,\gamma}=\operatorname{span}\Big\{\frac{\chi_{\gamma e}}{\mu[\gamma e]}
-\frac{\chi_{\gamma e'}}{\mu[\gamma e']}\,:\, (e,e')\in \operatorname{ext}_1(\gamma)\Big\}.
\end{equation}

We investigate for which values of $s$ the eigenvalue $\lambda_{s,\gamma}$ goes to $-\infty$ as $|\gamma|\to \infty$.
Note that this behavior has not been investigated before.

\begin{prop}\label{prop:eigenvalue}
For any $\delta\in (0,1)$, let $w_\delta$ be the weight given in \eqref{eq:weight-k}
and $d_{w_\delta}$ be the associated ultrametric given in Proposition~\ref{pr:weight-ultrametric} on $\PB_\Lambda$.
Let $\mu$ be the probability measure on $\partial \BB_\Lambda$ given in \eqref{eq:measure_kBD}.
Let $\{\rho_i: 1\le i\le k\}$ be the set of the spectral radii of all vertex matrices of $\Lambda$,
and suppose that $\rho_i>1$ for all $1\le i\le k$.
Let $\Delta_s, s\in \R$ be the Laplace-Beltrami operator associated to the Dirichlet form $Q_s$ given in \eqref{eq:Dirichlet-spectral}
and let $\lambda_{s,\gamma}$ be its eigenvalues given in \eqref{eq:ev1}, where $\gamma\in F\mathcal{B}_\Lambda$.
If $s<2+\delta$, then  for $\gamma\in F\mathcal{B}_\Lambda$,
the eigenvalue $\lambda_{s,\gamma}$ goes to $-\infty$ as $|\gamma|\to \infty$.
\end{prop}

\begin{proof}

Fix a finite path $\gamma$ in $\BB_\Lambda$.
For simplicity, we drop the subscript $s$ from $\lambda_{s,\gamma}$ for the proof.
In order to simplify the computation, we write the formula of $\lambda_\gamma$ in \eqref{eq:ev1} as
\[
\lambda_\gamma=A-B,
\]
where
\[
A=\sum_{n=0}^{|\gamma|-1}\frac{\mu[\gamma_{n+1}]-\mu[\gamma_n]}{G_s(\gamma_n)},\quad\quad B=\frac{\mu[\gamma]}{G_s(\gamma)}.
\]
Using the formulas for the measure $\mu$ and the weight $w_\delta$ on $\PB$,
we compute
\[
A =\sum_{n=0}^{|\gamma|-1}\frac{2 \big(w_\delta(\gamma_n)\big)^{s-2} (\mu[\gamma_{n+1}]-\mu[\gamma_n])}
{\sum_{(e,e')\in \operatorname{ext}_1(\gamma_n)} \mu[\gamma_n e]\cdot \mu[\gamma_n e']}.
\]
If $n=qk+t$ with $0\le t\le k-1$, then we have that
\begin{align*}
\mu[\gamma_n]&=(\rho_1\dots \rho_t)^{-(q+1)}(\rho_{t+1}\dots \rho_k)^{-q}\kappa^\Lambda_{s(\gamma_n)}, \\
\mu[\gamma_{n+1}]&=(\rho_1\dots \rho_{t+1})^{-(q+1)}(\rho_{t+2}\dots \rho_k)^{-q}\kappa^\Lambda_{s(\gamma_{n+1})}, \\
\mu[\gamma_n e]&=(\rho_1\dots \rho_{t+1})^{-(q+1)} (\rho_{t+2}\dots \rho_k)^{-q}\kappa^\Lambda_{s(e)}, \\
\mu[\gamma_n e']&=(\rho_1\dots \rho_{t+1})^{-(q+1)} (\rho_{t+2}\dots \rho_k)^{-q}\kappa^\Lambda_{s(e')}, \\
\big(w_\delta(\gamma_n)\big)^{s-2}&=\Big((\rho_1\dots \rho_t)^{-(q+1)/\delta}
(\rho_{t+1}\dots \rho_k)^{-q/\delta} \kappa^\Lambda_{s(\gamma_n)}\Big)^{s-2}.
\end{align*}
Thus, we obtain that
\[\begin{split}
A=2\sum_{n=0}^{|\gamma|-1} &(\rho_1\dots \rho_t)^{-(q+1)}  (\rho_{t+1}\dots \rho_k)^{-q}
\Big(\frac{\kappa^\Lambda_{s(\gamma_{n+1})}}{\rho_{t+1}}-\kappa^\Lambda_{s(\gamma_n)}\Big) \\
& \times \frac{\Big((\rho_1\dots \rho_t)^{-(q+1)/\delta} (\rho_{t+1}\dots \rho_k)^{-q/\delta}
\kappa^\Lambda_{s(\gamma_n)}\Big)^{s-2}}{\displaystyle \sum_{(e,e')\in \operatorname{ext}_1(\gamma_n)}
(\rho_1\dots \rho_{t+1})^{-2(q+1)} (\rho_{t+2}\dots \rho_k)^{-2q}\kappa^\Lambda_{s(e)} \kappa^\Lambda_{s(e')}},
\end{split}\]
so that
\[
A=2 \sum_{n=0}^{|\gamma|-1}\big(\frac{\kappa^\Lambda_{s(\gamma_{n+1})}}{\rho_{t+1}}-\kappa^\Lambda_{s(\gamma_n)}\Big)
\frac{(\rho_1\dots \rho_t)^{-\frac{(q+1)(s-2)}{\delta}+(q+1)}
(\rho_{t+1}\dots \rho_k)^{-\frac{q(s-2)}{\delta}+q} (\kappa^\Lambda_{s(\gamma_n)})^{s-2}}{\rho_{t+1}^{-2}
\displaystyle\sum_{(e,e')\in \operatorname{ext}_1(\gamma_n)} \kappa^\Lambda_{s(e)} \kappa^\Lambda_{s(e')}}.
\]

To compute $B$, let $|\gamma|=q'k+t'$ with $0\le t'\le k-1$. Then we have
\[\begin{split}
B&=\frac{\mu[\gamma] \, 2 \big(w_\delta(\gamma)\big)^{s-2}}{\sum_{(e,e')\in \operatorname{ext}_1(\gamma)}
  \, \mu[\gamma e]\cdot \mu[\gamma e']} \\
&=\frac{ 2 (\rho_1\dots \rho_{t'})^{-(q'+1)} (\rho_{t'+1}\dots \rho_k)^{-q'}\kappa^\Lambda_{s(\gamma)}
 \Big( (\rho_1\dots \rho_{t'})^{-\frac{(q'+1)}{\delta}} (\rho_{t'+1}\dots \rho_k)^{-\frac{q'}{\delta}}
 \kappa^\Lambda_{s(\gamma)}\Big)^{s-2}}{\sum_{(e,e')\in \operatorname{ext}_1(\gamma)}
 \, (\rho_1\dots \rho_{t'+1})^{-2(q'+1)} (\rho_{t'+2}\dots \rho_k)^{-2q'} \kappa^\Lambda_{s(e)} \kappa^\Lambda_{s(e')}}\\
&=\frac{2 (\rho_1\dots \rho_{t'})^{-\frac{(q'+1)(s-2)}{\delta}+(q'+1)}
 (\rho_{t'+1}\dots \rho_k)^{-\frac{q'(s-2)}{\delta}+q'} (\kappa^\Lambda_{s(\gamma)})^{s-1}}{\rho^{-2}_{t'+1}
 \sum_{(e,e')\in \operatorname{ext}_1(\gamma)} \, \kappa^\Lambda_{s(e)} \kappa^\Lambda_{s(e')}}.
\end{split}\]
Thus, we obtain that
\[\begin{split}
\lambda_\gamma & =A-B \\
 & =  2 \sum_{n=0}^{|\gamma|-1}
  \Big(\frac{\kappa^\Lambda_{s(\gamma_{n+1})}}{\rho_{t+1}}-\kappa^\Lambda_{s(\gamma_n)}\Big)
  \frac{(\rho_1\dots \rho_t)^{-\frac{(q+1)(s-2)}{\delta}+(q+1)} (\rho_{t+1}\dots \rho_k)^{-\frac{q(s-2)}{\delta}+q}
  (\kappa^\Lambda_{s(\gamma_n)})^{s-2}}{\rho_{t+1}^{-1}
  \sum_{(e,e')\in \operatorname{ext}_1(\gamma_n)} \, \kappa^\Lambda_{s(e)} \kappa^\Lambda_{s(e')}}  \\
&\qquad  -\frac{2 (\rho_1\dots \rho_{t'})^{-\frac{(q'+1)(s-2)}{\delta}+(q'+1)}
 (\rho_{t'+1}\dots \rho_k)^{-\frac{q'(s-2)}{\delta}+q'}
 (\kappa^\Lambda_{s(\gamma)})^{s-1}}{\rho^{-2}_{t'+1} \sum_{(e,e')\in \operatorname{ext}_1(\gamma)}
 \, \kappa^\Lambda_{s(e)} \kappa^\Lambda_{s(e')}}. \\
\end{split}\]

 We first note that $\lambda_\gamma$ are eigenvalues of a non-positive operator $\Delta_s$, we must have
\[
\frac{\kappa^\Lambda_{s(\gamma_{n+1})}}{\rho_{t+1}}-\kappa^\Lambda_{s(\gamma_n)}<0\quad\text{or}\quad \frac{\kappa^\Lambda_{s(\gamma_{n+1})}}{\rho_{t+1}}-\kappa^\Lambda_{s(\gamma_n)}=0
\]
for all $n$.

Now suppose that $\frac{\kappa^\Lambda_{s(\gamma_{n+1})}}{\rho_{t+1}}-\kappa^\Lambda_{s(\gamma_n)}<0$.
Then there are three cases to consider.

\textbf{Case 1}: If $2<s<2+\delta$, then $s-2>0$ and $\frac{(s-2)}{\delta}<1$. We have
\[
-(q+1) \big(\frac{s-2}{\delta}\big) + (q+1) = (q+1)\big(1-\frac{(s-2)}{\delta}\big)>0,
\]
so that
\[
(\rho_1\cdots \rho_t)^{-(q+1)(\frac{s-2}{\delta})+(q+1)}\to \infty\quad\text{as}\quad q\to \infty
\]
since $\rho_i>1$ for all $i\in \{1,\dots,k\}$. Similarly, we have
\[
-q(\frac{s-2}{\delta}) + q > 0,
\]
which implies that
\[
(\rho_{t+1}\cdots \rho_k)^{-q(\frac{s-2}{\delta})+q}\to\infty\quad\text{as}\quad q\to\infty.
\]
Therefore, we have that $A\to -\infty$ as $|\gamma|\to \infty$.
A similar argument also shows that $B\to \infty$ as $|\gamma| \to \infty$,
and hence $\lambda_\gamma=A-B\to -\infty$ as $|\gamma|\to \infty$.

\textbf{Case 2}: If $s<2$, then $s-2<0$. So, we have
\[
-(q+1)(\frac{s-2}{\delta})+(q+1)>0,
\]
and hence one can show that $\lambda_\gamma\to -\infty$ as $|\gamma|\to \infty$.

\textbf{Case 3}: If $s=2$, then we get
\[
-(q+1)(\frac{s-2}{\delta})+(q+1)=q+1>0.
\]
Thus, a similar argument shows that $\lambda_\gamma\to -\infty$ as $|\gamma|\to \infty$.
Therefore if $s<2+\delta$, then the eigenvalue $\lambda_\gamma\to -\infty$ as $|\gamma|\to\infty$.
\end{proof}

Because of Proposition~\ref{prop:eigenvalue},
the results in \cite{Kig2010} imply that  we can find an ultrametric $d^{(s)}$,
called the intrinsic metric in \cite{Kig2010}, on $\partial \BB_\Lambda$
associated to the eigenvalues of $\Delta_s$ as follows.

\begin{prop}\label{prop:intrinsic_metric}
Let $\Delta_s$ and $\lambda_{s,\gamma}$ be given in Proposition~\ref{prop:eigenvalue}
and $\{\rho_i: 1\le i\le k\}$ be the set of the spectral radii of all vertex matrices of $\Lambda$.
Suppose that  $\rho_i>1$ for all $1\le i\le k$.
\begin{itemize}
\item[\rm{(a)}] If $s<2$, then $\lambda_{s, \alpha}<0$ for any $\alpha \in F\BB_\Lambda$ with $|\alpha|\ge 1$, and
$\{\lambda_{s,x(0,n)}: n\in \N\}$ is strictly decreasing for any $x\in \partial \BB_\Lambda$, i.e.
\[
\lambda_{s, x(0,n)}> \lambda_{s,x(0,n+1)}.
\]
\item[\rm{(b)}] Fix $s\in \R$, and we define
\[
d^{(s)}(x,y)=\begin{cases} 0 & \text{for} \;\; x=y\in \partial \BB_\Lambda , \\
-\dfrac{1}{\lambda_{s, x\wedge y}} & \text{for $x\ne y \in \partial \BB_\Lambda$}. \end{cases}
\]
Then $d^{(s)}$ is an ultrametric on $\partial \BB_\Lambda$.

\item[\rm{(c)}]  Let the open ball of center $x\in \partial \BB_\Lambda$ and radius $a>0$
for the metric $d^{(s)}$ be defined by
\[
B_s(x,a)=\{y \mid d^{(s)}(x,y)<a\}.
\]
Then for each $x\in \partial \BB_\Lambda$, we have $B_s(x,a)=[x(0,n)]$ if and only if
\[
-\frac{1}{\lambda_{s, x(0,n)}}< a \le -\frac{1}{\lambda_{s,x(0,n-1)}}.
\]
\end{itemize}
\end{prop}

\begin{proof}
 A straightforward computation gives (a) so we leave it to the readers.
Also, (b) and (c) follow by similar arguments given in the proofs of Proposition~6.4(1)(3) of \cite{Kig2010}.

\end{proof}

\section{Dirichlet forms and heat kernels}
\label{sec:heat-kernel}

According to  \cite{Kig2010}, we can prove various interesting results for the metric measure space
that has a volume doubling property with respect to an ultrametric on the space.
The measure $\mu$ on $\PB_\Lambda$ has a volume doubling property
with respect to the both ultrametric $d^{(s)}$ and $d_{w_\delta}$ on $\PB_\Lambda$.
The former will be proved in Proposition~\ref{prop:VD_ds}
and the latter will become clear when we show Theorem~\ref{thm:asymp-metrics}.
Moreover, we show that there exists a heat kernel $p_t$ of a process associated to the Dirichlet form $Q_s$
and discuss the asymptotic behavior of $p_t$ after.

\subsection{Volume doubling property}
\label{subsec:VD}

For  a metric measure space $(X,\mu,d)$, we define an open ball with radius $r$ by
\[
B(x,r)=\{ y\in X\mid d(x,y)<r\}
\]
for $x\in X$ and $r>0$.
We say that the measure $\mu$ has the \emph{volume doubling property} with respect to a metric $d$
if there exists a constant $c>0$ such that
\[
\mu(B(x, 2r))\le c \cdot \mu(B(x,r))
\]
for any $x\in X$ and any $r>0$.

We have shown in the previous section that there are two metrics $d^{(s)}$ and $d_{w_\delta}$
defined on the measure space $(\partial \BB_\Lambda, \mu)$
if the spectral radii $\rho_i$ of vertex matrices of $\Lambda$ satisfy $\rho_i>1$ for $1\le i\le k$.
The volume doubling property of $\mu$ with respect to $d_{w_\delta}$ is essentially included
in the proof of Theorem \ref{thm:asymp-metrics}.
We first show  that the measure $\mu$ has the volume doubling property with respect to $d^{(s)}$  as below.

\begin{prop}\label{prop:VD_ds}
Suppose that the spectral radius $\rho_i$ of vertex matrices of $\Lambda$ satisfies $\rho_i>1$ for all $1\le i\le k$.
Let $\mu$ be the probability measure on $\partial \BB_\Lambda$ given in \eqref{eq:measure_kBD}
and $d^{(s)}$ be the intrinsic ultrametric on $\partial \BB_\Lambda$ given in Proposition~\ref{prop:intrinsic_metric}.
Then $\mu$ has the volume doubling property with respect to $d^{(s)}$ for $s<2$.
\end{prop}

\begin{proof}
Since $s<2$, first note that sequence of eigenvalues $\{\lambda_{s, x(0,n)}\}$ is strictly decreasing
for each $x\in \partial \BB_\Lambda$ and the associated intrinsic metric $d^{(s)}$ is an ultrametric on $\partial \BB_\Lambda$
by Proposition~\ref{prop:intrinsic_metric}.
We apply Theorem~6.5 of \cite{Kig2010} to the metric measure space $(\partial \BB_\Lambda, \mu, d^{(s)})$
in order to prove the proposition. So it suffices to prove two things.
First, we need to find $c_1\in (0,1)$ such that
\begin{equation}\label{eq:c_1}
c_1\le \frac{\mu[x(0,n)]}{\mu[x(0,n-1)]}\quad\text{for all $x\in \partial \BB_\Lambda$ and $n\in \N$},
\end{equation}
and second, we need to show that there exist $m\ge 1$ and $c_2\in (0,1)$ such that
\begin{equation}\label{eq:c_2}
\frac{\lambda_{s,x(0,n)}}{\lambda_{s, x(0,n+m)}}<c_2 \quad \text{for all $x\in \partial \BB_\Lambda$ and $n\in \N$.}
\end{equation}
To see the first claim, fix $x\in \partial \BB_\Lambda$ and $n\in \N$.
Then let $|x(0,n)|=n=qk+t$ for some $q\in \N$ and $0\le t \le k-1$.
Then we have that
\[\begin{split}
\mu[x(0,n)]&=\Big(\frac{1}{\rho^q \rho_1 \dots \rho_t}\Big) \kappa_{s(x(0,n))}^\Lambda,\\
\mu[x(0,n-1)]&=\Big(\frac{1}{\rho^q \rho_1 \dots \rho_{t-1}}\Big) \kappa_{s(x(0,n))}^\Lambda,
\end{split}\]
so that we get
\[
\frac{\mu[x(0,n)]}{\mu[x(0,n-1)]}=\frac{1}{\rho_t}\cdot \frac{\kappa^\Lambda_{s(x(0,n))}}{\kappa^\Lambda_{s(x(0,n-1))}} >0.
\]
Since the right-hand side of above equation is positive, there exists $c_1\in (0,1)$ that satisfies \eqref{eq:c_1}.

For the second claim, recall that the sequence of eigenvalues $\{\lambda_{s, x(0,n)}\}$
is strictly decreasing for each $x\in \partial \BB_\Lambda$ if $s<2$.
Thus, we have that
\[
0< \frac{\lambda_{s, x(0,n)}}{\lambda_{s, x(0,n+m)}}<1 \quad \text{for all $m\ge 1$}.
\]
Hence there exists $c_2\in (0,1)$ that satisfies \eqref{eq:c_2}.
Therefore $\mu$ has the volume doubling property with respect to $d^{(s)}$ if $s<2$.
\end{proof}

\subsection{Kernels and their asymptotic behaviors}

We show in this section that the Dirichlet form $Q_s$ coincides with the Dirichlet form $Q_{J_s,\mu}$
associated to a jump kernel $J_s$ in Proposition~\ref{prop:Dirichlet-J},
and show that there exists a heat kernel associated to the Dirichlet form $Q_s$ in Proposition~\ref{prop:heat-kernel-ex}.
Then we discuss the asymptotic behavior of the heat kernel
in Proposition~\ref{prop:p_bound}, Theorem~\ref{thm:asymp_p_s} and Theorem~\ref{thm:asymp-d-w}.

\begin{defn}[Definition~10.6 \cite{Kig2010}]
\label{def:Dirichlet-jump}
Let $\mu$ be the probability measure on $\partial \BB_\Lambda$ given in \eqref{eq:measure_kBD}.
Suppose that the spectral radius $\rho_i$ of vertex matrices of $\Lambda$ satisfies $\rho_i>1$ for all $1\le i\le k$
and that $J$ is a non-negative function on $F \mathcal{B}_\Lambda$.
We define $W_J:(\partial \BB_\Lambda \times \partial \BB_\Lambda) \to [0,\infty)$
by $W_J(x,y)=J(x\wedge y)$ for $x\ne y\in \partial \BB_\Lambda$.
Then we let
\[
\mathcal{D}_{J,\mu}=\{f\in L^2(\Lambda^\infty, \mu)\; :\;
\int_{\partial \BB_\Lambda \times \partial \BB_\Lambda} W_J(x,y) (f(x)-f(y))^2 \,d\mu(x) d\mu(y) <\infty\}
\]
and  for $f,g\in \mathcal{D}_{J,\mu}$, we define the bilinear form by
\[
\mathcal{Q}_{J,\mu}(f,g)
=\int_{\partial \BB_\Lambda \times \partial \BB_\Lambda} W_J(x,y) (f(x)-f(y)) (g(x)-g(y))\, d\mu(x) d\mu(y).
\]
The Dirichlet form of the above form is called a \textit{jump type Dirichlet form}
and the corresponding kernel $J$ is called a \textit{generalized jump kernel}.\footnote{Note that the definition works
for an arbitrary tree as long as there exists a measure on the associated infinite path space.
See also Definition~10.6 of \cite{Kig2010}.}
\end{defn}

As in \cite{Kig2010}, we now identify a generalized jump kernel $J_s$
for the Dirichlet form $Q_s$ of \eqref{eq:Dirichlet-spectral} as follows.
Note that the result below is more general than the one  in Section 13 of \cite{Kig2010}
that only concerns a tree case. Furthermore, our proof is involved with the spectral triple
while there is no spectral triple involved in \cite{Kig2010}.

\begin{prop}\label{prop:Dirichlet-J}
Let $\mu$ be the probability measure on $\partial \BB_\Lambda$ given in \eqref{eq:measure_kBD}.
Suppose that the spectral radius $\rho_i$ of vertex matrices of $\Lambda$ satisfies $\rho_i>1$ for all $1\le i\le k$.
For a finite path $\gamma\in\Lambda$, we let
\begin{equation}\label{eq:J_s}
J_s(\gamma)=\frac{(w_\delta(\gamma))^{s-2}}{\sum_{(e,e')\in \text{ext}_1} \mu[\gamma e] \, \mu[\gamma e']}.
\end{equation}
 The Dirichlet form $(Q_s, \operatorname{Dom}(Q_s))$ on $L^2(\partial \BB, \mu)$
given in \eqref{eq:Dirichlet-spectral} coincides with $(\mathcal{Q}_{J_s, \mu}, \mathcal{D}_{J_s, \mu})$ given in Definition~\ref{def:Dirichlet-jump}.
\end{prop}

\begin{proof}
We prove the proposition by computing the Dirichlet form $Q_s$
given in \eqref{eq:Dirichlet-spectral} explicitly as follows.
Note that $|D|=\sqrt{D^*D}$ and the Dirac operator $D$ acts
on the Hilbert space $\mathcal{H}=\ell^2(F\mathcal{B}_\Lambda)\otimes \C^2$.
For $\xi\in \mathcal{H}$ and $\lambda\in F\mathcal{B}_\Lambda$, we have
\[
|D|\xi(\lambda)=\frac{1}{w_\delta(\lambda)} \begin{pmatrix} 0 & 1 \\ 1 & 0 \end{pmatrix}
\begin{pmatrix} 0 & 1 \\ 1 & 0 \end{pmatrix} \xi(\lambda)
=\frac{1}{w_\delta(\lambda)} \begin{pmatrix} 1 & 0 \\ 0 & 1 \end{pmatrix} \xi(\lambda).
\]
For any $s\in \R$, we get
\[
|D|^{-s}=w_\delta(\lambda)^s \begin{pmatrix} 1 & 0 \\ 0 & 1 \end{pmatrix}.
\]
Similarly, we compute
\[
[D, \pi_\phi(g)]\xi(\lambda)=\frac{g(\phi_1(\lambda))-g(\phi_2(\lambda))}{(w_\delta(\lambda))^2}
\begin{pmatrix} 0 & -1 \\ 1 & 0 \end{pmatrix} \xi(\lambda),
\]
\[
[D,\pi_{\phi}(f)]^*\xi(\lambda)
=\frac{\overline{f}(\phi_1(\lambda))-\overline{f}(\phi_2(\lambda))}{(w_\delta(\lambda))^2}
\begin{pmatrix} 0 & 1 \\ -1 &  0 \end{pmatrix} \xi(\lambda).
\]
So, $\operatorname{Tr}\big(|D|^{-s} [D, \pi_\phi(f)]^* [D, \pi_\phi(g)]\big)$ is given by
\[
2\sum_{\lambda\in F\mathcal{B}_\Lambda} w_\delta(\lambda)^{s-2}
(\overline{f}(\phi_1(\lambda))-\overline{f}(\phi_2(\lambda))) (g(\phi_1(\lambda))-g(\phi_2(\lambda))).
\]
Thus, we have that
\begin{equation}\label{eq:Q_s-sum}
\begin{split}
Q_s(f,g) &=\frac{1}{2}\int_{\Xi} \operatorname{Tr}(|D|^{-s} [D, \pi_\phi(f)]^* [D, \pi_\phi(g)]) \, d\nu(\phi)\\
&= \int_{\Xi}\sum_{\lambda\in F \mathcal{B}_\Lambda} w_\delta(\lambda)^{s-2}
(\overline{f}(\phi_1(\lambda))-\overline{f}(\phi_2(\lambda))) (g(\phi_1(\lambda))-g(\phi_2(\lambda)))\, d\nu(\phi).
\end{split}\end{equation}

Now we apply the formula of $\nu$ on the set $\Xi$ of the choice functions given in \eqref{eq:measure_nu}.
Then we obtain that
\[
Q_s(f,g)=\int_{\partial \BB_\Lambda \times \partial \BB_\Lambda}
\frac{(w_\delta( x\wedge y))^{s-2}}{\sum_{(e,e')\in \operatorname{ext}_1(x \wedge y)}
\mu[(x\wedge y)e] \, \mu[(x\wedge y) e']} \big(\overline{f}(x)-\overline{f}(y)\big) \big(g(x)-g(y)\big) \, d\mu \, d\mu,
\]
where $x, y \in \partial \BB_\Lambda$ and $x\wedge y$ is the longest common path of $x$ and $y$.
(Note that a choice function $\phi$ only picks up a finite path $\lambda=x \wedge y$
for $(x,y)\in \partial \BB_\Lambda \times \partial \BB_\Lambda$,
so that the summation in $\lambda$ of \eqref{eq:Q_s-sum} goes away as above).
Therefore, by letting
\[
J_s(\gamma)=\frac{(w_\delta(\gamma))^{s-2}}{\sum_{(e,  e')\in \text{ext}_1(\gamma)} \mu[\gamma e] \, \mu[\gamma e']},
\]
we have shown that
\[
Q_s(f,g)=\mathcal{Q}_{J_s, \mu}(f,g).
\]
Also, it is straightforward to check that their domains coincide.
\end{proof}

Recall that each eigenspace $E_{s,\gamma}$ corresponding to the eigenvalue $\lambda_{s,\gamma}$
of the  non-positive definite self-adjoint Laplace-Beltrami operator
$\Delta_s$ associated to the above Dirichlet form $Q_s$ is given in \eqref{eq:eigenspace}:
 \[
E_{s,\gamma}=\operatorname{span}\Big\{\frac{\chi_{\gamma e}}{\mu[\gamma e]}-\frac{\chi_{\gamma e'}}{\mu[\gamma e']}\,
:\, (e,e')\in \operatorname{ext}_1(\gamma)\Big\}\subset L^2(\partial \mathcal{B}_\Lambda,\mu),
\]
which can be realized as
\[
\Big\{ \psi\in L^2(\partial \BB_\Lambda, \mu) :
\psi=\sum_{\substack{\alpha=\gamma e, \\ e\in s(\gamma)F\BB_\Lambda,\\ |e|=1}}
a_{\alpha} \chi_{[\alpha]}, \;\;  \sum_{\substack{\alpha=\gamma e,
\\ e\in s(\gamma)F\BB_\Lambda, \\ |e|=1}} a_\alpha=0 \Big\}.
\]
Let $\{\psi_{\gamma,1}, \dots, \psi_{\gamma, m_\gamma}\}$ be an $L^2(\partial \BB_\Lambda, \mu)$-orthonormal basis
of the above $E_{s,\gamma}$. Then Lemma~10.2 of \cite{Kig2010} implies that there exists a complete orthonormal system
$\{\psi_0, \psi_{\gamma, n} : \gamma\in F\BB_\Lambda, 1\le n\le m_\gamma\}$ of $L^2(\partial \BB_\Lambda, \mu)$,
where $\psi_0=\chi_{[\partial \BB]}$, $m_\gamma=|\{e\in s(\gamma) \BB^1\}|-1$ and $\BB^1$ is the set of edges in $\BB$.
Therefore, Lemma~7.1 of \cite{Kig2010} gives the formula
\[
\sum_{j=1}^{m_\gamma-1} \psi_{\gamma, j}(x) \psi_{\gamma, j}(y)
=\sum_{\substack{e\in s(\gamma)\BB^1}} \Big(\frac{\chi_{[\gamma e]}(x)\chi_{[\gamma e]}(y)}{\mu[\gamma e]}
-\frac{\chi_{[\gamma]}(x)\chi_{[\gamma]}(y)}{\mu[\gamma]}\Big),
\]
which is the same form as the formula given in the Section~7 of \cite{Kig2010}.

Now we are ready to give the formula of the heat kernel which we are interested in.
As in the equation (7.1) of \cite{Kig2010}, the heat kernel associated to the Dirichlet form
$(Q_s, \operatorname{Dom}(Q_s))$ is given by
\begin{equation}\label{eq:p-density}
p(t,x,y) = 1 + \sum_{\gamma\in F\mathcal{B}_\Lambda}e^{\lambda_{s,\gamma} t}
 \sum_{j=1}^{m_\gamma-1} \psi_{\gamma,j}(x) \, \psi_{\gamma,j}(y),
\end{equation}
where $\{\psi_0, \psi_{\gamma, n} : \gamma\in F\BB_\Lambda, 1\le n\le m_\gamma\}$ is a complete orthonormal system
of $L^2(\partial \BB_\Lambda, \mu)$ given as above.

In particular, we obtain that
\begin{equation}\label{eq:heat kernel_p}
p(t, x, y)=\begin{cases} 1 + \displaystyle \sum_{n=0}^\infty \Big( \frac{1}{\mu[x(0,n+1)]}-\frac{1}{\mu[x(0,n)]}\Big)
e^{\lambda_{s, x(0,n)}\, t} &  \text{if $x=y$}, \\
\displaystyle \sum_{n=0}^{|x \wedge y|} \frac{1}{\mu[x\wedge y(0,n)]}
\Big( e^{\lambda_{s, x\wedge y(0,n-1)}t} - e^{\lambda_{s, x\wedge y (0,n)} t}\Big) & \text{if $x\ne y$},
\end{cases}
\end{equation}
where $|x\wedge y|$ is the length of the path $x\wedge y$.
Note that we have $e^{\lambda_{s,\gamma}t}$ instead of
$e^{-\lambda_{s,\gamma}t}$ since $\lambda_{s, \gamma}<0$ for $\gamma\in F\BB_\Lambda$ in our case.

Since the measure $\mu$ on $\partial \BB_\Lambda$ has the volume doubling property
with respect to the intrinsic metric $d^{(s)}$ and the ultrametric $d_{w_\delta}$
induced by the weight $w_\delta$ on $\partial \BB$ given in \eqref{eq:weight-k},
we expect to obtain similar results to those in \cite[Section~7]{Kig2010},
such as asymptotic behaviors of heat kernel and jump kernel  of the process
associated to the Dirichlet form $Q_s$ given in \eqref{eq:Dirichlet-spectral}.
But, we first note that if the eigenvalues of the associated Laplace-Beltrami operator $\Delta_s$ blow up at infinity,
then one can find a Hunt process $(\{Y_t\}_{t>0}, \{P_x\}_{x\in\PB})$ on $\PB_\Lambda$ whose transition density is $p(t,x,y)$ as follows.

\begin{prop}
\label{prop:heat-kernel-ex}
Let $\mu$ be the probability measure on $\partial \BB_\Lambda$ given in \eqref{eq:measure_kBD}.
For any $x,y\in \PB_\Lambda$ and $t>0$, define $p^{t,x}(y)=p(t,x,y)$,
where $p(t,x,y)$ is the heat kernel given in \eqref{eq:p-density}
that is associated to the Dirichlet form $Q_s$ in \eqref{eq:Dirichlet-spectral}.
\begin{itemize}
\item[\rm{(a)}] For any bounded Borel measurable function $f:\PB_\Lambda\to \R$,
we define
\[
(p_t f)(x)=\int_{\PB_\Lambda} p^{t,x} f\, d\mu \qquad (t>0).
\]
Then $\{p_t: t>0\}$ is a Markovian transition function in the sense of \cite[Section~1.4]{FOT}.
\item[\rm{(b)}] There exists a Hunt process $(\{Y_t\}_{t>0}, \{P_x\}_{x\in \partial \BB_\Lambda})$ on $\partial \BB_\Lambda$
whose transition density  is $p(t, x, y)$, i.e.,
\[
{\mathbb E}_x(f(Y_t))=\int_{\partial \BB_\Lambda} p(t, x, y) f( y) \, \mu(dy),
\]
for $x\in \partial \BB_\Lambda$ and for a bounded Borel measurable function $f:\partial \BB_\Lambda \to \R$,
where ${\mathbb E}_x(\cdot)$ is the expectation with respect to $P_x$.
\end{itemize}
\end{prop}

\begin{proof}
Since $|\lambda_{s,\gamma}|\to \infty$ as $|\gamma|\to \infty$,
we see that the result follows from Proposition~7.2 and Theorem~7.3 of \cite{Kig2010}.
\end{proof}

Since $\{\lambda_{s,x(0,n)}: n\in \N\}$ is strictly decreasing
for any $x\in \partial \BB_\Lambda$ and the measure $\mu$ has the volume doubling property with respect to $d^{(s)}$,
the heat kernel satisfies the estimates in terms of  the intrinsic metric $d^{(s)}$ as follows.
Note that the proofs of Proposition~\ref{prop:p_bound}
and Theorem~\ref{thm:asymp_p_s} are very similar to the ones in \cite{Kig2010}.

\begin{prop}\label{prop:p_bound}(c.f. Proposition~7.5 of \cite{Kig2010})
Let $\mu$ be the probability measure on $\partial \BB_\Lambda$ given in \eqref{eq:measure_kBD}
and $p(t,x,y)$ be the heat kernel given in \eqref{eq:p-density}.
Suppose that $s<2$ and the spectral radius $\rho_i$ of vertex matrices of $\Lambda$ satisfies $\rho_i>1$ for every $1\le i \le k$.
Let $d^{(s)}$ be the intrinsic metric on $\partial \BB_\Lambda$ given in Proposition~\ref{prop:intrinsic_metric}
and we denote $B_s(x,t)$ be the open ball with radius $t$ centered at $x\in \partial \BB_\Lambda$ with respect to $d^{(s)}$.
Then the following statements are true.
\begin{itemize}
\item[\rm{(a)}] For $x\in \partial \BB_\Lambda$ and $t>0$, we have
~$p(t, x,x)\ge  \dfrac{1}{e}\cdot \dfrac{1}{\mu(B_s(x, t))}.$
\item[\rm{(b)}] For $0 < t \le d^{(s)}(x,y)$, we have
~$p(t,x,y)\le \dfrac{1}{d^{(s)}(x,y)\, \mu[x\wedge y]}.$
\end{itemize}
\end{prop}

\begin{proof}
The results follow from Proposition~7.5 of \cite{Kig2010}.
\end{proof}

Now we give the heat kernel estimations in terms of the intrinsic metric $d^{(s)}$ as follows.
We first recall that $f(x)\asymp g(x)$ means that there exist two positive numbers $c_1, c_2$ such that
\begin{equation}\label{eq:defn-asymp}
c_1 g(x) \le f(x) \le c_2 g(x).
\end{equation}

\begin{thm}\label{thm:asymp_p_s}(c.f. Theorem~7.6 of \cite{Kig2010})
Suppose that the spectral radius $\rho_i$ of vertex matrices of $\Lambda$ satisfies $\rho_i>1$ for all $1\le i \le k$.
Let $\mu$ be the probability measure on $\partial \BB_\Lambda$ as given in \eqref{eq:measure_kBD}
and $d^{(s)}$ be the intrinsic ultrametric on $\partial \BB_\Lambda$ given in Proposition~\ref{prop:intrinsic_metric}.
 Let $p(t,x,y)$ be the heat kernel given in \eqref{eq:p-density}.
Then the heat kernel $p(t,x,y)$ is continuous on $(0,\infty)\times \partial \BB_\Lambda \times \partial \BB_\Lambda$ and satisfies
\[
p(t,x,y) \asymp \begin{cases} \dfrac{t}{d^{(s)}(x,y) \, \mu[x \wedge y]} \quad & \text{if $0 < t \le d^{(s)}(x,y)$}, \\
 \dfrac{1}{\mu(B_s(x,t))}   & \text{if $t>d^{(s)}(x,y)$}.\end{cases}
\]
\end{thm}

\begin{proof}
We leave the proof to the reader since it is very similar to the proof in \cite[Theorem~7.6]{Kig2010}.
\end{proof}

We show in the main theorem of this paper that the ultrametric $d_{w_\delta}$
associated to the weights $w_\delta$ on $\BB_\Lambda$ is equivalent to
the intrinsic metric $d^{(s)}$ associated to the eigenvalues $\lambda_{s,\gamma}$ of $\Delta_s$.

\begin{thm}\label{thm:asymp-metrics}
Suppose that the spectral radius $\rho_i$ of vertex matrices $\Lambda$ satisfies $\rho_i>1$ for every $1\le i \le k$.
Let $d_{w_\delta}$ be the ultrametric on $\partial \mathcal{B}_\Lambda$
associated to the weight $w_\delta$ given in \eqref{eq:weight-k} for $\delta\in (0,1)$,
and let $d^{(s)}$ be the intrinsic metric given in Proposition~\ref{prop:intrinsic_metric}.
If $1< s < 2+\delta$, then we have that
\begin{equation}\label{eq:d-asymp}
d^{(s)}(x,y)\asymp (d_w(x,y))^{2+\delta-s}
\end{equation}
for $x,y \in \partial \BB_\Lambda$.

\end{thm}

\begin{proof}
We have to find two positive constants $a, b$ such that
\[
a\cdot d^{(s)}(x,y) \le (d_w(x,y))^{2+\delta-s} \le b\cdot d^{(s)}(x,y).
\]
To see this, fix a finite path $\gamma\in F\BB_\Lambda$ with $|\gamma|=n=qk+t$
where $q\in \N$ and $1\le t\le k-1$. Then for $s\in \R$, \eqref{eq:ev1} gives
\begin{equation}
\lambda_{s,\gamma}=\sum_{k=0}^{n-1}\frac{\mu[\gamma_{k+1}]-\mu[\gamma_k]}{G_s(\gamma_k)}-\frac{\mu[\gamma]}{G_s(\gamma)},
\end{equation}
where $G_s(\eta)=\frac{1}{2} \big(w(\eta)\big)^{2-s}\sum_{(e,e')\in \text{ext}_1(\eta)} \mu[\eta e] \, \mu[\eta e']$.
In particular, for any finite path $\eta$ with length $|\eta|=m=pk+\ell$, one can compute
\[
G_s(\eta)=\big(\rho_1^{p+1} \dots \rho_{\ell}^{p+1} \rho_{\ell+1}^{p} \dots \rho_k^{p}\big)^{\frac{s-2}{\delta}-2}
(\kappa^\Lambda_{s(\eta)})^{2-s} \, (2 (\rho_{\ell+1})^{-2} )
\sum_{(e,e')\in \text{ext}_1(\eta)} \kappa^\Lambda_{s(e)} \, \kappa^\Lambda_{s(e')}.
\]

To simplify the computations,
the first sum and the second sum of $\lambda_{s,\gamma}$ are denoted by $A$ and  $B$, respectively.
Then we compute
\[\begin{split}
B &=\frac{\mu[\gamma]}{G_s(\gamma)}
 = \frac{(\rho_1^{q+1}\dots \rho_t^{q+1} \rho_{t+1}^q\dots \rho_k^q)^{-1}\kappa^\Lambda_{s(\gamma)}}
 {(\rho_1^{q+1}\dots \rho_t^{q+1} \rho_{t+1}^q\dots \rho_k^q)^{\frac{s-2}{\delta}-2}\cdot (\kappa^\Lambda_{s(\gamma)})^{2-s}
 \cdot \frac{\sum_{(e,e') \in \text{ext}_1(\gamma)} \kappa^\Lambda_{s(e)} \kappa^{\Lambda}_{s(e')}}{2(\rho_{t+1})^2}}\\
&=(\rho_1^{q+1}\dots \rho_t^{q+1} \rho_{t+1}^q\dots \rho_k^q)^{1-\frac{s-2}{\delta}}\cdot
 (\kappa^\Lambda_{s(\gamma)})^{s-1} (\alpha_\gamma)^{-1},
\end{split}\]
where $\alpha_\gamma= \frac{\sum_{(e,e') \in \text{ext}_1(\gamma)} \kappa^\Lambda_{s(e)} \kappa^{\Lambda}_{s(e')}}{2(\rho_{t+1})^2}$.

To compute $A$, now let $\gamma_i$ be a sub-path of $\gamma$ with $|\gamma_i|=i=q'k+t'$. Then
\[\begin{split}
A&=\sum_{i=0}^{n-1}\frac{\mu[\gamma_{i+1}]-\mu[\gamma_i]}{G_s(\gamma_i)}\\
&=\sum_{i=1}^{n-1} \frac{ (\rho_1^{q'+1} \dots \rho_{t'}^{q'+1} \rho_{t'+1}^{q'+1} \rho_{t'+2}^{q'}\dots \rho_k^{q'})^{-1}\cdot
 \kappa^\Lambda_{s(\gamma_{i+1})} - (\rho_1^{q'+1}\dots \rho_{t'}^{q'+1} \rho_{t'+1}^{q'}\dots \rho_k^{q'})^{-1} \cdot
 \kappa^\Lambda_{s(\gamma_i)}}{(\rho_1^{q'+1}\dots \rho_{t'}^{q'+1} \rho_{t'+1}^{q'}\dots \rho_k^{q'})^{\frac{s-2}{\delta}-2} \cdot
 (\kappa^\Lambda_{s(\gamma_i)})^{2-s} \cdot (\alpha_{\gamma_i})},
\end{split}\]
where $\alpha_{\gamma_i}=\frac{\sum_{(e,e')\in \text{ext}_1(\gamma_i)} \kappa^\Lambda_{s(e)} \kappa^\Lambda_{s(e')}}{2(\rho_{t'+1})^2}$.
Then we can simplify $A$ and obtain
\[
A=\sum_{i=1}^{n-1}(\rho_1^{q'+1}\dots \rho_{t'}^{q'+1} \rho_{t'+1}^{q'}\dots \rho_k^{q'})^{\frac{2+\delta-s}{\delta}}
 \cdot (\kappa^\Lambda_{s(\gamma_i)})^{s-2} (\alpha_{\gamma_i})^{-1}
 \cdot \big(\rho^{-1}_{t'+1}\cdot \kappa^\Lambda_{s(\gamma_{i+1})}-\kappa^\Lambda_{s(\gamma_i)}\big).
\]
We see that $\lambda_{s,\gamma}=A-B$ is given by
\[
\begin{split}
( & \rho_1^{q+1} \dots \rho_{t}^{q+1} \rho_{t+1}^{q}\dots \rho_k^{q})^{\frac{2+\delta-s}{\delta}}
 \cdot \Big(\sum_{i=1}^{n-1} (\rho_1^{q'+1}\dots \rho_{t'}^{q'+1} \rho_{t'+1}^{q'}\dots \rho_k^{q'})^{\frac{2+\delta-s}{\delta}}\\
& \cdot (\rho_1^{q+1}\dots \rho_{t}^{q+1} \rho_{t+1}^{q}\dots \rho_k^{q})^{-\frac{2+\delta-s}{\delta}}
\cdot (\kappa^\Lambda_{s(\gamma_i)})^{s-2} (\alpha_{\gamma_i})^{-1}
 (\rho_{t'+1}^{-1}\kappa^\Lambda_{s(\gamma_{i+1})}-\kappa^\Lambda_{s(\gamma_i)})
 -(\kappa^\Lambda_{s(\gamma)})^{s-1}\cdot (\alpha_\gamma)^{-1} \Big).
\end{split}
\]
Therefore, we have $d^{(s)}(x,y)=-(\lambda_{s, x\wedge y})^{-1}=-(A-B)^{-1}=(B-A)^{-1}$
 if $\gamma=x\wedge y$, and hence we get
\begin{equation}\label{eq:dsF}
d^{(s)}(x,y)= (\rho_1^{q+1} \dots \rho_{t}^{q+1} \rho_{t+1}^{q}\dots \rho_k^{q})^{\frac{2+\delta-s}{\delta}}\cdot F_\gamma,
\end{equation}
where $F_\gamma$ is a finite sum given by
\begin{equation}\label{eq:formula_F}
\begin{split}
F_\gamma=\bigg( \frac{(\kappa^\Lambda_{s(\gamma)})^{s-1}}{\alpha_\gamma}
 - \sum_{i=1}^{n-1}\big(\rho_1^{q'+1}\dots \rho_{t'}^{q'+1} \rho_{t'+1}^{q'}\dots & \rho_k^{q'}\big)^{\frac{2+\delta-s}{\delta}}
\big(\rho_1^{q+1}\dots \rho_{t}^{q+1} \rho_{t+1}^{q}\dots \rho_k^{q}\big)^{-\frac{2+\delta-s}{\delta}} \\
&\cdot \frac{(\kappa^\Lambda_{s(\gamma_i)})^{s-2}}{\alpha_{\gamma_i}} \cdot
 \big(\frac{\kappa^\Lambda_{s(\gamma_{i+1})}}{\rho_{t'+1}}-\kappa^\Lambda_{s(\gamma_i)}\big) \bigg)^{-1} .
\end{split}
\end{equation}

Since $d^{(s)}(x,y)>0$, $F_\gamma$ is positive.
We notice that $F_\gamma$ only involves with a finite number of terms.
Thus, there exists an upper bound of $F_\gamma$.
In particular, there exists the least upper bound of $F_\gamma$, denoted by $\sup_{\gamma} F_\gamma$, which is not zero.
Since $d_w(x,y)= (\rho_1^{q+1} \dots \rho_{t}^{q+1} \rho_{t+1}^{q}\dots \rho_k^{q})^{-\frac 1{\delta}} \cdot \kappa^\Lambda_{s(\gamma)}$,
the equation \eqref{eq:dsF} implies
\begin{equation}\label{eq:ds-dw relation}
d^{(s)}(x,y)=(d_w(x,y))^{\frac{2+\delta-s}{\delta}} \cdot \frac{F_\gamma}{(\kappa^\Lambda_{s(\gamma)})^{2+\delta-s}}
\end{equation}
where  $\gamma=x\wedge y$ with $|\gamma|=n=qk+t$.
This gives the equality
\[
\frac{(\kappa_{s(\gamma)}^\Lambda)^{2+\delta-s}}{F_\gamma}\, d^{(s)}(x,y)= (d_w(x,y))^{2+\delta-s}.
\]

We claim that there exists the minimum value of $\frac{(\kappa_{s(\gamma)}^\Lambda)^{2+\delta-s}}{F_\gamma}$
that does not depend on $s$.
Indeed, we observe that
\[
\frac{(\kappa_{s(\gamma)}^\Lambda)^{2+\delta-s}}{F_\gamma} \ge \frac{(\kappa_{s(\gamma)}^\Lambda)^{2+\delta-s}}{\sup_\gamma F_\gamma}.
\]
We also note that $2+\delta -s >0$ and $0< \kappa_{s(\gamma)}^\Lambda <1$.
Since there exists the minimum value of a finite set of positive numbers and $\Lambda$ is a finite $k$-graph,
the set $\{(\kappa_{s(\gamma)}^\Lambda)^{2+\delta-s}: s(\gamma)\in \Lambda^0\}$ has the minimum value,
denoted by $\min \{(\kappa_{s(\gamma)}^\Lambda)^{2+\delta-s}\}$.
Thus we have
\[
\frac{(\kappa_{s(\gamma)}^\Lambda)^{2+\delta-s}}{F_\gamma} \ge
\frac{\min\{(\kappa_{s(\gamma)}^\Lambda)^{2+\delta-s}\}}{\sup_\gamma F_\gamma}.
\]
By letting $a:=\frac{\min\{(\kappa_{s(\gamma)}^\Lambda)^{2+\delta-s}\}}{\sup_\gamma F_\gamma}$, we obtain $a\,d^{(s)}(x,y) \le (d_w(x,y))^{2+\delta-s}$.

To obtain the other inequality, note that \eqref{eq:ds-dw relation} gives
\[
d_w(x,y)^{2+\delta-s} = d^{(s)}(x,y) \, \frac{(\kappa_{s(\gamma)}^\Lambda)^{2+\delta-s}}{F_\gamma}.
\]
Since $0 < \kappa_v^\Lambda<1$ and $2+\delta-s >0$, we have
\[
d_w(x,y)^{2+\delta-s} = d^{(s)}(x,y) \, \frac{(\kappa_{s(\gamma)}^\Lambda)^{2+\delta-s}}{F_\gamma}< d^{(s)}(x,y) \frac{1}{F_\gamma}.
\]
Thus, we only need to show that there exists an upper bound
of the set $\{\frac{1}{F_\gamma}: \gamma \in F\mathcal{B}\}$ that does not depend on $s$.
From \eqref{eq:formula_F}, we see that
\[
\begin{split}
\frac{1}{F_\gamma}=\frac{(\kappa^\Lambda_{s(\gamma)})^{s-1}}{\alpha_\gamma}
 - \sum_{i=1}^{n-1}(\rho_1^{q'+1}\dots \rho_{t'}^{q'+1} \rho_{t'+1}^{q'}\dots \rho_k^{q'})^{\frac{2+\delta-s}{\delta}}
& (\rho_1^{q+1}\dots \rho_{t}^{q+1} \rho_{t+1}^{q}\dots \rho_k^{q})^{-\frac{2+\delta-s}{\delta}}\\
&\cdot \frac{(\kappa^\Lambda_{s(\gamma_i)})^{s-2}}{\alpha_{\gamma_i}} \cdot
 \bigg(\frac{\kappa^\Lambda_{s(\gamma_{i+1})}}{\rho_{t'+1}}-\kappa^\Lambda_{s(\gamma_i)}\bigg).
\end{split}
\]

We note that each $\gamma_i$ with $|\gamma_i|=q'k+t'$  is a sub-path of $\gamma$ with $|\gamma|=qk+t$.
Since $\frac{2+\delta-s}{\delta} >0$ and $\rho_i>1$ for all $1\le i \le k$, we have that
\[
\bigg( \frac{\rho_1^{q'+1}\dots \rho_{t'}^{q'+1} \rho_{t'+1}^{q'}\dots \rho_k^{q'}}
{\rho_1^{q+1}\dots \rho_{t}^{q+1} \rho_{t+1}^{q}\dots \rho_k^{q}}\bigg)^{\frac{2+\delta-s}{\delta}} <1.
\]
Thus, we have that
\[
\begin{split}
\frac{1}{F_\gamma}
&=\frac{(\kappa^\Lambda_{s(\gamma)})^{s-1}}{\alpha_\gamma}
 + \sum_{i=1}^{n-1}\frac{(\rho_1^{q'+1}\dots \rho_{t'}^{q'+1} \rho_{t'+1}^{q'}\dots
 \rho_k^{q'})^{\frac{2+\delta-s}{\delta}}}{(\rho_1^{q+1}\dots \rho_{t}^{q+1} \rho_{t+1}^{q}\dots \rho_k^{q})^{\frac{2+\delta-s}{\delta}}}
 \cdot \frac{(\kappa^\Lambda_{s(\gamma_i)})^{s-2}}{\alpha_{\gamma_i}} \cdot
 \Big(\kappa^\Lambda_{s(\gamma_i)} - \frac{\kappa^\Lambda_{s(\gamma_{i+1})}}{\rho_{t'+1}}\Big) \\
& \le \frac{(\kappa^\Lambda_{s(\gamma)})^{s-1}}{\alpha_\gamma}
 + \sum_{i=1}^{n-1}\frac{(\kappa^\Lambda_{s(\gamma_i)})^{s-2}}{\alpha_{\gamma_i}} \cdot \Big(\kappa^\Lambda_{s(\gamma_i)}-\frac{\kappa^\Lambda_{s(\gamma_{i+1})}}{\rho_{t'+1}}\Big) \\
&= \frac{2 (\rho_{t+1})^2 (\kappa_{s(\gamma)}^\Lambda)^{s-1}}{\sum_{(e,e')\in \operatorname{ext}_1(\gamma)}
 \kappa_{s(e)}^\Lambda \kappa_{s(e')}^\Lambda} +\sum_{i=1}^{n-1} \frac{2(\rho_{t'+1})^2
 (\kappa_{s(\gamma_i)}^\Lambda)^{s-2}}{\sum_{(e,e')\in \operatorname{ext}_1(\gamma_i)} \kappa_{s(e)}^\Lambda \kappa_{s(e')}^\Lambda}
 \cdot \Big(\kappa^\Lambda_{s(\gamma_i)}-\frac{\kappa^\Lambda_{s(\gamma_{i+1})}}{\rho_{t'+1}}\Big).
\end{split}
\]

Since $\Lambda$ has only finitely many vertices,
there exists the positive minimum value
of the set $\{\sum_{(e,e')\in \operatorname{ext}_1(\eta)} \kappa_{s(e)}^{\Lambda} \kappa_{s(e')}^\Lambda : \eta\in F\BB_\Lambda\}$
denoted by $\operatorname{MIN}$.
Also, there exists the maximum value of $\{\rho_j:1\le j \le k\}$, which we denote by $\rho_{\max}$.
Thus, we have that
\[
\begin{split}
\frac{1}{F_\gamma}
& < \frac{2 (\rho_{\max})^2}{ \operatorname{MIN}}
\bigg( (\kappa_{s(\gamma)}^\Lambda)^{s-1}+\sum_{i=1}^{n-1} (\kappa_{s(\gamma_i)}^\Lambda)^{s-2} \Big(\kappa^\Lambda_{s(\gamma_i)}-\frac{\kappa^\Lambda_{s(\gamma_{i+1})}}{\rho_{t'+1}}\Big) \bigg) \\
& = \frac{2 (\rho_{\max})^2}{ \operatorname{MIN}}
\bigg( (\kappa_{s(\gamma)}^\Lambda)^{s-1}+\sum_{i=1}^{n-1} (\kappa_{s(\gamma_i)}^\Lambda)^{s-1}
\Big(1 -\frac{\kappa^\Lambda_{s(\gamma_{i+1})}}{\rho_{t'+1}\cdot \kappa_{s(\gamma_i)}}\Big) \bigg).
\end{split}
\]
Since $0< \kappa_{s(\gamma)}^\Lambda <1$ and $s-1>0$,
we have $(\kappa_{s(\gamma)}^\Lambda)^{s-1}<1$ and hence
\[
\frac{1}{F_\gamma}<\frac{2(\rho_{\max})^2}{\operatorname{MIN}}\bigg(1+\sum_{i=1}^{n-1}(\kappa_{s(\gamma_i)}^\Lambda)^{s-1} (1-\frac{\kappa^\Lambda_{s(\gamma_{i+1})}}{\rho_{t'+1}\cdot \kappa_{s(\gamma_i)}^\Lambda})\bigg).
\]

If $\Big(1-\frac{\kappa^\Lambda_{s(\gamma_{i+1})}}{\rho_{t'+1}\cdot \kappa^\Lambda_{s(\gamma_i)}} \Big)<0$,
then we have the inequality $\frac{1}{F_\gamma}<\frac{2(\rho_{\max})^2}{\operatorname{MIN}}$,
which completes the proof.

If $0< \Big(1 -\frac{\kappa^\Lambda_{s(\gamma_{i+1})}}{\rho_{t'+1}\cdot \kappa^\Lambda_{s(\gamma_i)}}\Big) <1$
for any $i=0,1,\ldots,n-1$, then we have that
\[
\frac{1}{F_\gamma}< \frac{2 (\rho_{\max})^2}{ \operatorname{MIN}}
 \bigg(1+\sum_{i=1}^{n-1} (\kappa_{s(\gamma_i)}^\Lambda)^{s-1} \bigg).
\]
Since $\kappa_{s(\gamma_i)}^\Lambda$'s are entries of the unimodular Perron-Frobenius eigenvector and $s-1>0$,
$\sum_{i=1}^{n-1} (\kappa_{s(\gamma_i)}^\Lambda)^{s-1}$ is uniformly bounded by some constant $M$.
This gives the inequality
\[
\frac{1}{F_\gamma}  < \frac{2 (\rho_{\max})^2}{ \operatorname{MIN}}(1+M),
\]
which completes the proof.
\end{proof}

We showed that there exists a heat kernel $p(t,x,y)$ associated to the Dirichlet form $Q_s$
in Proposition~\ref{prop:heat-kernel-ex} and described its asymptotic behavior
in Proposition~\ref{prop:p_bound} and Theorem~\ref{thm:asymp_p_s} with respect to the intrinsic metric $d^{(s)}$.
Due to the previous theorem, we can now give the heat kernel estimates
in terms of the ultrametric $d_{w_\delta}$ associated to the weight $w_\delta$
given in \eqref{eq:weight-k} on $\partial \BB_\Lambda$.
In fact, (b) of the following theorem may be thought of
as a heat kernel estimate for jump processes associated to the Dirichlet form $\mathcal{Q}_{J_s, \mu}$.

\begin{thm}
\label{thm:asymp-d-w}
Suppose that the spectral radius $\rho_i$ of vertex matrices of $\Lambda$ satisfies $\rho_i>1$ for every $1\le i \le k$.
Let $d_{w_\delta}$ be the ultrametric on $\partial \mathcal{B}_\Lambda$
associated to the weight $w_\delta$ given in \eqref{eq:weight-k} for $\delta\in (0,1)$.
Suppose that $1\le s< 2 + \delta$. Then we have the following facts.
\begin{itemize}
\item[\rm{(a)}] There exists a jointly continuous transition density $p(t,x,y)$
on $(0,\infty)\times \partial \BB_\Lambda\times \partial \BB_\Lambda$
for the Hunt process $(\{Y_t\}_{t>0}, \{P_x\}_{x\in \partial \BB_\Lambda})$
associated to the Dirichlet form $(\mathcal{Q}_{J_s, \mu}, \mathcal{D}_{J_s, \mu})$
on $L^2(\partial \BB_\Lambda, \mu)$ given in \eqref{prop:Dirichlet-J}.

\item[\rm{(b)}] The transition density $p(t, x, y)$ in \rm{(a)} satisfies
\[
p(t, x, y) \asymp
\begin{cases} \dfrac{t}{d_{w_\delta}(x,y)^{\frac{2+\delta-s}{\delta}}}
 & \text{if}\;\; d_{w_\delta}(x,y)^{\frac{2+\delta-s}{\delta}}>t  \\
\dfrac{1}{\mu(B_{d_{w_\delta}}(x, t^{\frac{\delta}{2+\delta-s}}))}
 & \text{if}\;\;  d_{w_\delta}(x,y)^{\frac{2+\delta-s}{\delta}} \le t
\end{cases}
\]
for any $(t, x, y)\in (0,\infty)\times \partial \BB_\Lambda\times \partial \BB_\Lambda$.
\item[(c)] For any $x\in \partial \BB_\Lambda$ and $t\in (0,1]$,
\[
\mathbb{E}_x(d_{w_\delta}(x, Y_t)^{\frac{2+\delta-s}{\delta}\alpha}) \asymp
\begin{cases} t & \text{if}\;\; \alpha>1, \\
t (|\log t|+1) & \text{if}\;\; \alpha=1, \\
t^\gamma & \text{if} \;\; 0<\alpha<1.
\end{cases}
\]
\end{itemize}
\end{thm}

\begin{proof}
Note that (a) and (b) follow from Theorem~\ref{thm:asymp_p_s} and Theorem~\ref{thm:asymp-metrics}.
Also (c) follows by Corollary~7.9 of \cite{Kig2010}
since the measure $\mu$ has the volume doubling property with respect to the ultrametric $d_{w_\delta}$.
\end{proof}

\end{document}